\documentclass[11pt]{amsart}
\usepackage[all]{xy}
\usepackage{amssymb}
\usepackage{enumitem}
\usepackage{amsthm}
\usepackage{colortbl}
\usepackage[dvipsnames]{xcolor}
\usepackage{amsmath}
\usepackage{amscd,enumitem}
\usepackage{verbatim}
\usepackage{eurosym}
\usepackage{float}
\usepackage{color}
\usepackage{rotating}
\usepackage{dcolumn}
\usepackage[mathscr]{eucal}
\usepackage[all]{xy}
\usepackage{bbm}
\usepackage{makecell}
\usepackage{wasysym}
\usepackage[margin=1in]{geometry} 
\usepackage{tikz}
\usepackage{graphicx}
\newcommand\myhdots{\hbox to 1.5em{.\hss.\hss.}}
\usetikzlibrary{automata,arrows,positioning,calc}
\DeclareFontFamily{U}{wncy}{}
    \DeclareFontShape{U}{wncy}{m}{n}{<->wncyr10}{}
    \DeclareSymbolFont{mcy}{U}{wncy}{m}{n}
    \DeclareMathSymbol{\Sh}{\mathord}{mcy}{"58} 
\newtheorem*{thm*}{Theorem}
\newtheorem*{conj*}{Conjecture}

\newtheorem{theorem}{Theorem}[section]

\newtheorem{lemma}[theorem]{Lemma}

\newtheorem{proposition}[theorem]{Proposition}

\newtheorem*{remark*}{Remark}

\newtheorem{corollary}[theorem]{Corollary}

\theoremstyle{definition}
\newtheorem{definition}[theorem]{Definition}
\newtheorem{example}[theorem]{Example}

\newtheorem*{question}{Question}

\newcommand{\ZZ}{\mathbb{Z}}

\newcommand{\NN}{\mathbb{N}}

\numberwithin{equation}{section}

  {\list{}{\leftmargin=0.25in}\item[]}%
  {\endlist}

  \usepackage{tikz}

\usetikzlibrary{calc}
\tikzset{vtx/.style={circle, fill, inner sep=1.5pt}}
\tikzset{openvtx/.style={circle, draw, inner sep=1.5pt}}

\title{On the number of generalized numerical semigroups}
\author{Sean Li}
\address{Department of Mathematics, Massachusetts Institute of Technology, Cambridge, MA 02139, USA}
\email{seanjli@mit.edu}
\makeatletter
\renewcommand{\paragraph}[1]{%
  \par
  \addvspace{\medskipamount}
  \textit{#1\@addpunct{.}}\enspace\ignorespaces
}
\makeatother

\makeatletter
\renewcommand{\subparagraph}[1]{%
  \par
  \addvspace{\medskipamount}

  \noindent
  \textbf{#1\@addpunct{.}}\enspace\ignorespaces
}
\makeatother

\date{\today}

\usepackage{comment}

\usepackage{epigraph}
\usepackage{stmaryrd}
\usetikzlibrary{patterns}
\usetikzlibrary{decorations.pathmorphing}

\usepackage{mathrsfs}

\newcommand{\eps}{\varepsilon}

\newcommand{\vocab}{\emph}

\newcommand{\rr}{\mathsf{r}}
\newcommand{\cc}{\mathsf{c}}
\newcommand{\mm}{\mathcal{M}}

\newcommand{\bs}{\mathbf}
\usepackage{thmtools}
\usepackage{mathrsfs}

\usepackage{hyperref}
\usepackage{cleveref}

\begin{document}
\maketitle

\begin{abstract}
  Let $\mathsf{r}_k$ be the unique positive root of $x^k - (x+1)^{k-1} = 0$. We prove the best known bounds on the number $n_{g,d}$ of $d$-dimensional generalized numerical semigroups, in particular that
  \[n_{g,d} > C_d^{g^{(d-1)/d}} \rr_{2^d}^g\]
  for some constant $C_d > 0$, which can be made explicit. To do this, we extend the notion of multiplicity and depth to generalized numerical semigroups and show our lower bound is sharp for semigroups of depth 2. We also show other bounds on special classes of semigroups by introducing \emph{partition labelings}, which extend the notion of Kunz words to the general setting.
\end{abstract}

\section{Introduction}

Let $\NN_0$ denote the nonnegative integers. A \emph{numerical semigroup} $\Lambda$ is a subset of $\NN_0$ that has finite complement, contains 0, and is closed under addition. In other words, $\Lambda$ is a cofinite submonoid of $\NN_0$. Given a numerical semigroup $\Lambda$, one can define a number of invariants of $\Lambda$: its \vocab{genus} $g(\Lambda) := \# (\NN_0 \setminus \Lambda)$; its \vocab{multiplicity} $m(\Lambda) := \min \{ \Lambda \setminus \{0\} \}$; its \vocab{Frobenius number} $f(\Lambda) := \max \{ \NN_0 \setminus \Lambda \}$; and its \vocab{depth} $q(\Lambda) := \lceil (1+f(\Lambda))/m(\Lambda) \rceil$, the last of which was recently introduced by Eliahou and Fromentin~\cite{ef_2020}. We drop the $\Lambda$ when writing down these invariants if it is clear to which semigroup we are referring.

A great deal of research has been done regarding the enumeration of numerical semigroups after ordering by a specific invariant. Perhaps the most significant result is due to Zhai~\cite{zhai_2012}, who in 2011 showed that the number of numerical semigroups with genus $g$ is asymptotic to $C\left( \frac{1+\sqrt{5}}{2} \right)^g$ for some constant $C$, resolving a conjecture of Bras-Amor\'os~\cite{bras-amoros_2007}. The asymptotic number of numerical semigroups with fixed Frobenius number~\cite{backelin_1990}, multiplicity~\cite{kaplan_2011}, and recently depth with respect to genus~\cite{zhu_2022} and multiplicity~\cite{li_2022} are also known. A key ingredient in many of these proofs is that, after ordering, almost all numerical semigroups have small depth. For instance, Zhai showed that almost all numerical semigroups of genus $g$ have depth $2$ or $3$, verifying a conjecture of Zhao~\cite[Conj.~2]{zhao_2011}.

Naturally, one can ask similar questions about the analogous objects in $\NN_0^d$. A \vocab{generalized numerical semigroup} (or \vocab{GNS}) $\Lambda^d$ of dimension $d$ is a cofinite submonoid of $\NN_0^d$. In 2016, Failla, Peterson, and Utano~\cite{fpu_2017} initiated the formal study on the number $n_{g,d}$ of generalized numerical semigroups with genus $g$ and dimension $d$, showing that for fixed $d$ and large $g$ we have 
\[n_{g,d} \gtrsim \binom{g+d-1}{d-1} C^d \left( \frac{1+\sqrt{5}}{2} \right)^g\]
by estimating the number of generalized numerical semigroups that contain every point with positive integer coordinates. Their bound implies that $n_{g,d}^{1/g} \gtrsim \frac{1+\sqrt{5}}{2}$ for fixed $d$. To our knowledge, the above bound is the best known bound on $n_{g,d}$. Currently, the asymptotics of $n_{g,d}$ are wide open; even for $d=2$, we do not have a Bras-Amor\'os-like conjecture for the growth rate~\cite{kaplan_2017}. However, Cisto, Delgado, and Garc\'ia-S\'anchez~\cite{cdg_2021} have developed algorithms to compute $n_{g,d}$ for small values.

In this paper, we constructively show the following main result, which significantly improves the lower bound on $n_{g,d}$.

\begin{theorem}
  \label{thm:gns-record}
  Let $\rr_k$ be the unique positive root of $x^k - (x+1)^{k-1} = 0$. Then for each $d$, there is a constant $C_d > 0$ for which 
  \[n_{g,d} > C_d^{g^{(d-1)/d}} \rr_{2^d}^g.\]
\end{theorem}

In particular, we have that $n_{g,d}^{1/g} \gtrsim \rr_{2^d}$ for fixed $d$. Hence, \Cref{thm:gns-record} is stronger than the lower bound from Failla, Peterson, and Utano~for $d \geq 2$. This also implies that $\lim_{g\to\infty} (n_{g,1})^2/n_{g,2} = 0$, since $\left( \frac{1+\sqrt5}{2} \right)^2 < \rr_4$, negatively answering a question of Cisto, Delgado, and Garc\'ia-S\'anchez~\cite[$\S$8, pg.~16]{cdg_2021}. We furthermore establish the following superexponential upper bound on $n_{g,d}$.

\begin{theorem}
  \label{thm:bad-upper}
  For each $d$, we have $n_{g,d}^{1/g} < (2e+o(1)) (\ln g)^{d-1}$.
\end{theorem}

Analyzing $n_{g,d}$ is difficult in part because numerical semigroup invariants are not easy to lift to the general case. For instance, Failla, Peterson, and Utano define the multiplicity and Frobenius gap of a generalized numerical semigroup, but they are not canonical; they depend on a choice of relaxed monomial order on $\NN_0^d$~\cite[Def.~3.5]{fpu_2017}. To this end, Cisto, Failla, Peterson, and Utano define the \vocab{Frobenius allowable gaps}~\cite{cisto_2019}, which Lin and Singhal~\cite{lin_2021} showed to be the set of maximal elements of $\NN_0^d \setminus \Lambda^d$ under the natural partial order and used them to count semigroups with a unique Frobenius allowable gap.

In light of this, a key ingredient of our paper is our extension of multiplicity and depth to the general setting. We define the \vocab{multset} to be the set of nonzero minimal elements of $\Lambda^d$ under the natural partial order. Then, we generalize the notion of depth to generalized numerical semigroups and show that~\Cref{thm:gns-record} is sharp for generalized numerical semigroups of depth 2. An advantage of our more flexible definition is that it does not rely on a choice of relaxed monomial order on $\NN_0^d$.

An important tool used in the study of numerical semigroups is the Kunz word, which is a word with integer entries satisfying certain linear inequalities that encodes a numerical semigroup. This correspondence allows us to turn the problem of enumerating numerical semigroups into a problem in polyhedral geometry~\cite{kaplan_2011, kaplan_2021} and additive combinatorics~\cite{li_2022, zhu_2022}. 

We lift the notion of Kunz words to generalized numerical semigroups by introducing \emph{partition labelings}, which are diagrams labeled with multi-dimensional partitions that satisfy certain additive inequalities that encode the data of a GNS. We use partition labelings to bound the number of semigroups with a fixed multset and the number of semigroups with a multset of minimal size, the latter of which we call \vocab{rectangular GNSs}. 

\subsection{Outline}

In \Cref{sec:prelims}, we establish conventions and review the relevant technical facts about $d$-dimensional partitions and the constants $\rr_k$ for our paper. Next, in \Cref{sec:invariants}, we define the multset, depth, and depth-$k$ regions of a GNS, which generalize numerical semigroup invariants. In \Cref{sec:bounds}, we prove~\Cref{thm:gns-record} and~\Cref{thm:bad-upper}. Then in \Cref{sec:part}, we define partition labelings and use them to bound the size of special classes of semigroups. Finally, we discuss open questions and future lines of work in \Cref{sec:future}.

\section*{Acknowledgements}

This work was done at the University of Minnesota Duluth with support from Jane Street Capital, the National Security Agency (grant number H98230-22-1-0015), and fully supported by the CYAN Undergraduate Mathematics Fund at MIT. The author especially thanks Joseph Gallian for his mentorship and for nurturing a wonderful environment for research. We are also grateful to Michael Ren, Amanda Burcroff, and Deepesh Singhal for detailed comments, and Jonas Iskander and Noah Kravitz for helpful discussions. Finally, we are thankful to Carmelo Cisto and GAP for aiding with computations. 

\section{Preliminaries}

\label{sec:prelims}

In this section, we go over some background material on the structure of $\NN_0^d$ and $d$-dimensional partitions, then discuss the constants $\rr_k$ which appear throughout the paper.

\subsection{Points in $\NN_0^d$}

We denote points in $\NN_0^d$ by bolded lowercase letters and their coordinates by unbolded letters with subscripts, e.g., $\mathbf{x} = (x_1, \dots, x_d)$. The points in $\NN_0^d$ have a natural partial order $\leq$. Namely, we let $\mathbf{a} \leq \mathbf{b}$ if $a_i \leq b_i$ for $i = 1,\dots, d$, with equality if and only if $a_i = b_i$ for all $i$. Hence, the unique minimal element of $\NN_0^d$ is the origin $\mathbf{0} := (0, \dots, 0)$. 

Also, let $\mathbf{e^d_i}$ be the point in $\NN_0^d$ with $(e^d_i)_j = 0$ when $i \neq j$ and $(e^d_i)_j =1$ when $i = j$, so $\mathbf{e^d_1}, \dots, \mathbf{e^d_d}$ generate $\NN_0^d$ as an additive monoid.

\subsection{$d$-dimensional partitions}

A \vocab{$d$-dimensional partition} of $n$ is a partition $\pi$ into nonnegative integer parts $\pi_\bs{x}$, indexed by $\bs{x} \in \NN_0^d$, for which 
\[ n = \sum_{\mathbf{x} \in \NN_0^d} \pi_{\mathbf{x}} \qquad \text{and} \qquad \pi_{\mathbf{a}} \geq \pi_{\mathbf{b}} \text{ if }\mathbf{a} \geq \mathbf{b}.\]
For $d=0$, these consist of a single number; for $d=1$, these are the typical partitions of $n$. The cases $d=2$ and $d=3$ are known as \vocab{plane partitions} and \vocab{solid partitions}, respectively. 

One can visually represent a $d$-dimensional partition in $(d+1)$-dimensional space by stacking $\pi_{\bs{x}}$ hypercubes of dimension $d+1$ atop the axis-aligned unit $d$-cell whose least vertex (under the partial order) is at $\bs{x}$. This is known as the \vocab{Young diagram} of $\pi$ in the one-dimensional case.

Throughout the paper, we let $p_d(n)$ denote the number of $d$-dimensional partitions of $n$ and let $P_d(x) := \sum_{n=0}^\infty p_d(n)x^n$ be the $d$-dimensional partition generating function. Recall that
\begin{align*}
  P_0(x) &:= \sum_{n \geq 0} x^n = \frac{1}{1-x}, \\
  P_1(x) &:= \sum_{n \geq 0} p_1(n)x^n = \prod_{k=1}^\infty \frac{1}{1-x^k}.
\end{align*}
It is also a classic result of MacMahon~\cite{macmahon_1915} that
\[P_2(n) := \sum_{n \geq 0} p_2 (n)x^n = \prod_{k=1}^\infty \frac{1}{(1-x^k)^k}.\]
From these expressions, one can calculate the following asymptotic expressions for $p_1(n)$ and $p_2(n)$, first calculated in the 1900's by Hardy-Ramanujan~\cite{hr_1918} and Wright~\cite{wright_1931}, respectively:
\begin{align*}
  p_1(n) &\sim \frac{1}{4n\sqrt3} \exp\left( \pi \sqrt{\frac{2n}{3}} \right), \\
  p_2(n) &\sim \frac{\zeta(3)^{7/36}}{\sqrt{12\pi}} \left( \frac{n}2 \right)^{-25/36} \exp\left(3 \zeta(3)^{1/3} \left( \frac{n}2 \right)^{2/3} + \zeta'(-1)\right),
\end{align*}
where $\zeta$ denotes the Riemann zeta function.

Unfortunately, a closed-form expression for $P_d(n)$ for $d \geq 3$ and the precise asymptotics of $p_d(n)$ are unknown. However, of use to us is the following asymptotic result on $p_d(n)$ due to Arora, Bhatia, and Prasad~\cite{abp_1997}.

\begin{theorem}[{Arora-Bhatia-Prasad,~\cite[eq.~6]{abp_1997}}]
  \label{thm:abp}
  For each $d$, there exist $k_d^{-}, k_d^+ > 0$ for which
  \[k_d^{-} < \frac{\ln p_d(n)}{n^{d/(d+1)}} < k_d^+.\]
\end{theorem}

More information on multi-dimensional partitions can be found in~\cite{andrews_1998}. A possible point of confusion is the discrepancy between the dimension of a generalized numerical semigroup and the dimension of a multi-dimensional partition; when working with a $d$-dimensional GNS, we usually opt to work with $(d-1)$-dimensional partitions. To make this distinction clear, especially with indices, we use the convention $\bs{a}, \bs{b}, \bs{e_i^d}, \bs{m}, \bs{x}, \bs{y}, \bs{z} \in \NN_0^d$ and $\bs{e_i^{d-1}}, \bs{u}, \bs{v}, \bs{w} \in \NN_0^{d-1}$.

\subsection{Min-sums}

In this paper, we introduce the following operation on $d$-dimensional partitions.

\begin{definition}
  For two $d$-dimensional partitions $\pi$ and $\pi'$, the \vocab{min-sum} $\pi \boxplus \pi'$ is the $d$-dimensional partition $\tau$ with entries $\displaystyle \tau_{\mathbf{x}} = \min_{\mathbf{a} + \mathbf{b} = \mathbf{x}} \{ \pi_{\mathbf{a}} + \pi'_{\mathbf{b}} \}$.
\end{definition}

One can check that $\boxplus$ is commutative and associative. To our knowledge, this operation on multi-dimensional partitions, even in the one-dimensional case, is new.

\begin{example}
  Let $\pi = [4,2,1]$ and $\pi' = [3,2,2,1]$ be one-dimensional partitions. Then we have
  \begin{align*}
    (\pi \boxplus \pi')_0 &= 4 + 3 = 7\\
    (\pi \boxplus \pi')_1 &= \min(4+2,2+3) = 5 \\
    (\pi \boxplus \pi')_2 &= \min(4+2,2+2,1+3) = 4 \\
    &\vdots
  \end{align*}
  Overall, we get $\pi \boxplus \pi' = [7,5,4,3,2,2,1]$.
\end{example}

  We will use the min-sum operator $\boxplus$ to help us define the multset and depth of a GNS.

  \begin{remark*}
  There is a connection between $\boxplus$ and tropical geometry. Specifically, let $a \oplus b := \min(a,b)$ and $a \odot b := a + b$ denote tropical addition and multiplication, respectively. Then working in $\ZZ[x_1, \dots, x_d]$, we have that
  \[\left( \bigoplus_{\bs{a} \in \NN_0^d} (\pi_\bs{a} \odot x^\bs{a}) \right) \odot \left( \bigoplus_{\bs{a} \in \NN_0^d} (\pi'_\bs{a} \odot x^\bs{a}) \right) = \bigoplus_{\bs{a} \in \NN_0^d} (\pi \boxplus \pi')_\bs{a} \odot x^\bs{a},\]
  where $x^{\bs{a}} = x_1^{a_1} \cdots x_d^{a_d}$.

  We can also interpret the min-sum operator in terms of multiplication of monomial ideals. If we work in the ring $\ZZ[x_1, \dots, x_d, y]$ then we have that
  \[\left( \sum_{\bs{a} \in \NN_0^d} (x^{\bs{a}} y^{\pi_{\bs{a}}}) \right) \cdot \left( \sum_{\bs{a} \in \NN_0^d} (x^{\bs{a}} y^{\pi'_{\bs{a}}}) \right) = \sum_{\bs{a} \in \NN_0^d} (x^{\bs{a}} y^{(\pi \boxplus \pi')_{\bs{a}}}),\]
  where we are summing principal ideals.

\end{remark*}

\subsection{The constants $\rr_k$}
It will be helpful to define the following family of constants. 

\begin{definition}
For positive integers $k$, let the constant $\rr_k$ be the root of largest magnitude of $x^k - (x+1)^{k-1} = 0$. The values of these constants begin as follows:
\begin{align*}
  \rr_1 &= 1 \\
  \rr_2 &= \frac{1+\sqrt{5}}{2} = 1.6180\dots \\
  \rr_3 &= 2.1479\dots \\
  \rr_4 &= 2.6297\dots \\
  &\vdots
\end{align*}
\end{definition}

We verify that $\rr_k$ is well-defined with the following lemma.

\begin{lemma}
  The polynomial $x^k - (x+1)^{k-1}$ has a unique root of largest magnitude $\rr_k$, and $\rr_k \geq 1$.
\end{lemma}

\begin{proof}

Let $f(x) = x^k - (x+1)^{k-1}$. We prove the following facts, which suffice to prove the claim:
\begin{itemize}
  \item[(1)] We have that $\rr_k$ is real and a simple root.
  \item[(2)] The polynomial $f$ has a unique positive root $\rr_k^+$, and $\rr_k^+ > 1$.
  \item[(3)] We must have $\rr_k > 0$.
\end{itemize}

To prove (1), note that any root $r$ of $f$ satisfies $|r|^k = |r+1|^{k-1} \leq (|r| + 1)^{k-1}$ with equality only if $r$ is real, so any root will have magnitude at most the magnitude of the largest real root. Hence, we have that $\rr_k$ is real. Moreover, one can check that $f(x)$ and $f'(x)$ share no common roots, so $\rr_k$ is a simple root.

To show (2), note that $f(1) \leq 0$ and the leading coefficient of $f$ is positive, so $f$ has a positive root $\rr_k^+ \geq 1$. By Descartes' rule of signs, $f$ has a unique positive root. Thus, if $\rr_k > 0$ then $\rr_k = \rr_k^+$.

Finishing with (3), suppose for the sake of contradiction that $\rr_k < 0$. Then we must have that $\rr_k \leq -|\rr_k^+| \leq -1$ since $\rr_k$ is the root of largest magnitude. But note that for $x < -1$, we have $f(x) > 0$ when $k$ is even and $f(x) < 0$ when $k$ is odd. In particular, we have $f(\rr_k) \neq 0$, a contradiction.
\end{proof}

While not strictly necessary for the rest of the paper, it is still natural to consider the size of $\rr_k$ as $k$ goes to infinity. By writing the equation as $x^{k/(k-1)} = x+1$, it is clear that $\rr_k$ is increasing in $k$ and goes to infinity as $k$ goes to infinity.

\begin{proposition}
  We have $\rr_k \sim k/(\ln k)$.
\end{proposition}

\begin{proof}
  Rewrite the given equation to the form $x = (1+1/x)^{k-1}$, which in turn implies
  \begin{align*}
  \rr_k^{\rr_k} &= \left(\left( 1 + \frac{1}{\rr_k} \right)^{\rr_k}\right)^{k-1} = e^{k(1-o(1))},
\end{align*}
since as $k$ goes to infinity, so does $\rr_k$. The solution to the equation $x^x = e^a$ is $x = e^{W(a)}$, where $W$ is the Lambert $W$-function. It is known (see, e.g.~\cite[eq.~4.18]{w_1996}) that $W(x) = \ln x - \ln \ln x + o(1)$, so we have
\[\rr_k = e^{\ln k - \ln \ln k + o(1)} = (1+o(1))\left( \frac{k}{\ln k} \right),\]
as desired.
\end{proof}

\section{Invariants of generalized numerical semigroups}

\label{sec:invariants}

In this section, we define a number of invariants that will help us count generalized numerical semigroups.

\subsection{Multiplicity and depth}

First, we define analogues of multiplicity and depth for generalized numerical semigroups.

Recall that the \vocab{multiplicity} $m(\Lambda)$ of a numerical semigroup $\Lambda$ is the least nonzero integer contained in $\Lambda$. In higher dimensions, a GNS may not have a unique least nonzero point. Nevertheless, we generalize the notion of multiplicity to GNSs as follows.

\begin{definition}
  Let $\Lambda^d$ be a GNS. The \vocab{multset} $\mm(\Lambda^d)$ is the set of minimal, nonzero points in $\Lambda^d$. In other words, $\mm(\Lambda^d) := \{ \mathbf{m} \in \Lambda^d : \mathbf{m} \not > \mathbf{x} \text{ for all }\mathbf{x} \in \Lambda^d \setminus \{ \bs{0} \} \}$.
\end{definition}

As a reminder, we drop the $\Lambda^d$ when it is clear to which GNS we are referring. We use the calligraphic $\mm$ to remind the reader that the multset is a set, not an integer like in the one-dimensional case. However, note for numerical semigroups $\Lambda$, we have $\mm(\Lambda) = \{m(\Lambda)\}$. We have the following characterization of possible multsets.

\begin{proposition}
  \label{prop:mult-char}
  A finite set $\mm \subset \NN_0^d \setminus \{\bs{0}\}$ is a possible multset if and only if:
  \begin{itemize}
    \item it is an antichain of $\NN_0^d$ under $\leq$; and
    \item it contains some multiple of $\mathbf{e^d_i}$ for $i=1,\dots,d$. 
  \end{itemize}
\end{proposition}

\begin{proof}
  We first show the forward direction. If $\mm(\Lambda^d)$ contained two comparable elements $\bs{x} > \bs{y}$, then this contradicts the definition of a multset, since $\bs{y} \in \Lambda^d \setminus \{ \bs{0} \}$. On the other hand, if $\mm$ did not contain an element on the $i$-th coordinate axis, then $\Lambda^d$ does not contain any point on the axis and is thus not cofinite.

  In the reverse direction, suppose $\mm$ is an antichain that contains the elements $m_i \cdot \bs{e^d_i}$ for indices $i=1,\dots,d$. Let
  \[\Lambda^d = \{ \bs{x} \in \NN_0^d : \bs{x} \geq \bs{m} \text{ for some }\bs{m} \in \mm\}.\]
  Then $\Lambda^d$ is closed upwards, so it is closed under addition. It is cofinite, since it includes all elements with $x_i \geq m_i$ for each $i$ and thus excludes at most $\prod_{i=1}^d m_i$ elements. Finally, it has multset $\mm$ by construction. Hence, $\mm$ is a valid multset.
\end{proof}

We also generalize the notion of depth to GNSs. For sets $A,B \subseteq \NN_0^d$, define
\begin{align*}
  A+B &:= \{\mathbf{a} + \mathbf{b} : \mathbf{a} \in A, \ \mathbf{b} \in B\}, \\
  kA &:= \underbrace{A + \dots + A}_{k \text{ times}}.
\end{align*}
 In particular, $kA \neq \{ka : a \in A\}$. One may recognize these as Minkowski sum operators.

\begin{definition}
  The \vocab{depth} $q(\Lambda^d)$ of the GNS $\Lambda^d$ is the least integer $q$ for which
  \[\{\bs{x} \in \NN_0^d : \bs{x} \geq \bs{a} \text{ for some }\bs{a} \in q\mm\} \subseteq \Lambda^d.\]
\end{definition}

\begin{example}
  \label{ex:gns-2}
  Let 
  \[\Lambda^2 = \NN_0^2 \setminus \{(0,1),(0,3),(1,0),(1,1),(1,2),(2,0), (3,1), (3,3), (4,1)\}.\] 
  Then $\Lambda^2$ is a GNS with genus $g = 9$, multset $\mm = \{(0,2),(2,1),(3,0)\}$, and depth $q=3$.
  See the left side of~\Cref{fig:ex-gns}, where cells correspond to their bottom left corner (gray if included, white if excluded) and the stars indicate the multset.

  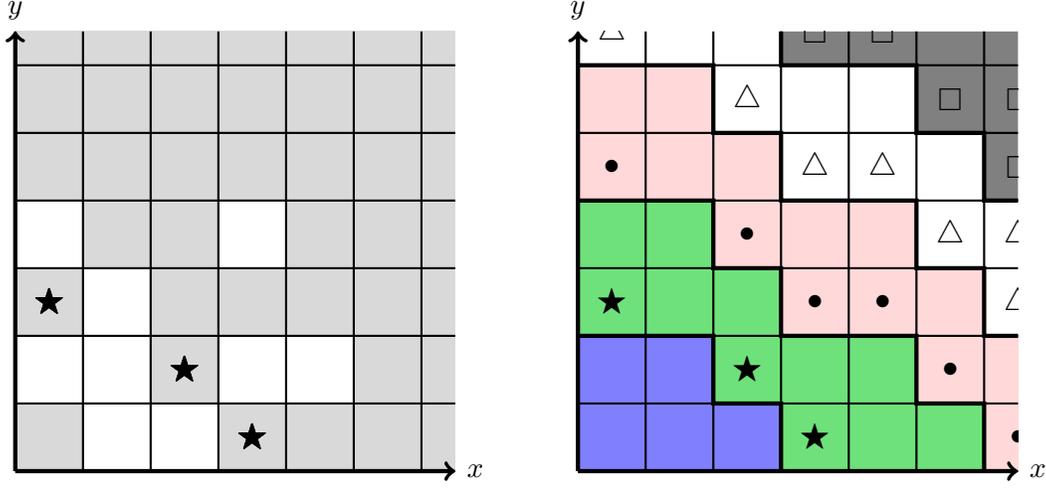
\begin{figure}
    \begin{center}
  \begin{tikzpicture}[scale=0.9]
    \fill[gray!30!white] (0,0)--(6.5,0)--(6.5,6.5)--(0,6.5)--cycle;
    \foreach \xx/\yy in {0/1,0/3,1/0,1/1,1/2,2/0,3/1,3/3,4/1} {
    \fill[white] (\xx,\yy)--(\xx+1,\yy)--(\xx+1,\yy+1)--(\xx,\yy+1)--cycle;
    }
    \draw[ultra thick,->] (0,0)--(6.5,0) node[right] {$x$};
    \draw[ultra thick,->] (0,0)--(0,6.5) node[above] {$y$};
    t
    \foreach \i in {0,...,6} {
    \draw[thick] (0,\i)--(6.5,\i);
    \draw[thick] (\i,0)--(\i,6.5);
    \foreach \xx/\yy in {0/2,2/1,3/0} {
    \node at (\xx+0.5,\yy+0.5) {\large $\bigstar$}; 
    }
    }
  \end{tikzpicture}
  \qquad
  \begin{tikzpicture}[scale=0.9]
    \foreach \xx/\hh in {0/1,1/1,2/0} {
    	\foreach \yy in {0,...,\hh} {
    \fill[blue!50!white] (\xx,\yy)--(\xx+1,\yy)--(\xx+1,\yy+1)--(\xx,\yy+1)--cycle;
	}
    }
    \foreach \xx/\bb/\hh in {0/2/3,1/2/3,2/1/2,3/0/1,4/0/1,5/0/0} {
    	\foreach \yy in {\bb,...,\hh} {
    \fill[red!30!blue!20!green!60!white] (\xx,\yy)--(\xx+1,\yy)--(\xx+1,\yy+1)--(\xx,\yy+1)--cycle;
	}
    }
    \foreach \xx/\bb/\hh in {0/4/5,1/4/5,2/3/4,3/2/3,4/2/3,5/1/2} {
    	\foreach \yy in {\bb,...,\hh} {
    \fill[red!15!white] (\xx,\yy)--(\xx+1,\yy)--(\xx+1,\yy+1)--(\xx,\yy+1)--cycle;
	}
    }
    \fill[gray] (3,6)--(5,6)--(5,5)--(6,5)--(6,4)--(6.5,4)--(6.5,6.5)--(3,6.5)--cycle;
    \fill[red!15!white] (6,0)--(6,2)--(6.5,2)--(6.5,0)--cycle;
    \foreach \i in {0,...,6} {
    \draw[thick] (0,\i)--(6.5,\i);
    \draw[thick] (\i,0)--(\i,6.5);
    }
    
    \begin{scope}
      \clip (6.5,0)--(6.5,6.5)--(0,6.5)--(0,0)--cycle;
    \foreach \xx/\yy in {0/2,2/1,3/0} {
    \node at (\xx+0.5,\yy+0.5) {\large $\bigstar$}; 
    }

    \foreach \xx/\yy in {0/4,2/3,3/2,4/2,5/1,6/0} {
    \node at (\xx+0.5,\yy+0.5) {\large $\bullet$}; 
    }
    
    \foreach \xx/\yy in {0/6,2/5,3/4,4/4,5/3,6/2,6/3,7/2,8/1,9/0} {
    \node at (\xx+0.5,\yy+0.5) {\large $\triangle$}; 
    }

    \foreach \xx/\yy in {3/6,5/5,6/4,4/6,6/5} {
    \node at (\xx+0.5,\yy+0.5) {\large $\square$}; 
    }
  \end{scope}

    \draw[ultra thick,->] (0,0)--(6.5,0) node[right] {$x$};
    \draw[ultra thick,->] (0,0)--(0,6.5) node[above] {$y$};
    \draw[ultra thick] (0,2)--(2,2)--(2,1)--(3,1)--(3,0);
    \draw[ultra thick] (0,4)--(2,4)--(2,3)--(3,3)--(3,2)--(5,2)--(5,1)--(6,1)--(6,0);
    \draw[ultra thick] (0,6)--(2,6)--(2,5)--(3,5)--(3,4)--(5,4)--(5,3)--(6,3)--(6,2)--(6.5,2);
    \draw[ultra thick] (3,6.5)--(3,6)--(5,6)--(5,5)--(6,5)--(6,4)--(6.5,4);
  \end{tikzpicture}
\end{center}
  \caption{A two-dimensional GNS and multset, along with its depth-$k$ regions.}
  \label{fig:ex-gns}
\end{figure}
\end{example}

\subsection{Depth-$k$ regions}

Fixing the multset of a GNS is a strong condition that imposes restrictions on what other elements must be included. In the one-dimensional case, we characterize numerical semigroups of multiplicity $m$ by partitioning 
\[\NN_0 = \{0,1,\dots,m-1\} \sqcup \{m,\dots,2m-1\} \sqcup \{2m, \dots 3m-1\} \sqcup\cdots\]
into sets of size $m$. Then the depth is the least $q$ for which $\{qm, \dots, (q+1)m-1\} \subseteq \Lambda$, since then every integer above $qm$ is contained in $\Lambda$. We somewhat generalize this concept to GNSs as follows.

\begin{definition}
  Let $\mm$ be a multset. For nonnegative integers $k$, define the region
  \begin{align*}
    \mathcal{R}_{\leq k}(\mm) &:= \{ \mathbf{x} \in \NN_0^d : \mathbf{x} \not\geq \mathbf{a} \text{ for all }\mathbf{a} \in k\mm\}.
\end{align*}
The \vocab{depth-$k$ region} $\mathcal{R}_k(\mm)$ is the set $\mathcal{R}_{\leq k}(\mm) \setminus \mathcal{R}_{\leq k-1}(\mm)$. The size of the depth-$k$ region is $s_k(\mm) := \# \mathcal{R}_k(\mm)$. By convention, we let $\mathcal{R}_0(\mm) = \{\bs{0}\}$.
\end{definition}

Note that depth-$k$ regions are defined in terms of multsets, not GNSs, though we still drop the $\mm$ when the argument is clear. These regions serve as a ``blueprint'' for a possible GNS of a given multset, where we first let $\Lambda^d$ have multset $\mm$ then choose to exclude certain elements from finitely many regions $\mathcal{R}_k(\mm)$.

\begin{lemma}
  \label{prop:dumb-stuff}
  If $\bs{x} \in \mathcal{R}_k$ and $\bs{y} \in \mathcal{R}_{\ell}$, then $\bs{x} + \bs{y} \not\in \mathcal{R}_{\leq k+\ell-2}$.
\end{lemma}

\begin{proof}
  There are elements $\bs{a} \in (k-1)\mm$ and $\bs{b} \in (\ell - 1)\mm$ for which $\bs{x} \geq \bs{a}$ and $\bs{y} \geq \bs{b}$. Then $\bs{x} + \bs{y} \geq \bs{a} + \bs{b} \in (k+\ell-2)\mm$, so $\bs{x}+\bs{y}$ cannot be in the region $\mathcal{R}_{\leq k+\ell-2}$.
\end{proof}

\begin{corollary}
  \label{cor:dumb}
  The depth of $\Lambda^d$ is the least integer $q$ for which $\mathcal{R}_{q+1}(\mm(\Lambda^d)) \subset \Lambda^d$.
\end{corollary}

\begin{proof}
  If $q$ is the depth of $\Lambda^d$, then by definition $\mathcal{R}_{q+1} \subset \Lambda^d$, so we show the reverse direction. Assume $\Lambda^d \neq \NN_0^d$, so $q \geq 2$. It suffices to show that if $\mathcal{R}_k \subset \Lambda^d$, then $\mathcal{R}_{k+1} \subset \Lambda^d$, since then every element bounded below by an element of $q\mm$ is in $\Lambda^d$.

  Suppose that $\bs{x} \in \mathcal{R}_{k+1}$, so there exists an $\bs{a} \in k\mm$ for which $\bs{x} \geq \bs{a}$. Then there exists an element $\bs{m} \in \mm \subseteq \mathcal{R}_2$ for which $\bs{a} - \bs{m} \in (k-1)\mm$. Now consider the point $\bs{x} - \bs{m}$. It cannot be in $\mathcal{R}_{\leq k-1}$, since $\bs{x} - \bs{m} \geq \bs{a} - \bs{m} \in (k-1)\mm$, and it cannot be in $\mathcal{R}_{k+1}$, since then by~\Cref{prop:dumb-stuff} we would have $\bs{m} + (\bs{x} - \bs{m}) \not\in \mathcal{R}_{\leq k+1}$. Hence, we have that $\bs{x} - \bs{m} \in \mathcal{R}_k$ is an element of $\Lambda^d$, so $\bs{x}$ is, too.
\end{proof}

Now, we relate the regions $\mathcal{R}_{\leq k}(\mm)$ with $(d-1)$-dimensional partitions in the following way. For $\mathbf{v} \in \NN_0^{d-1}$, let $(\pi^k)_\mathbf{v}(\mm)$ be the least integer $t$ for which $(v_1, \dots, v_{d-1}, t) \not\in \mathcal{R}_{\leq k}(\mm)$.

\begin{lemma}
  If $\mathbf{u} \leq \mathbf{v}$, then $(\pi^k)_\mathbf{u} \geq (\pi^k)_\mathbf{v}$.
\end{lemma}

\begin{proof}
  Suppose not. Then $(u_1, \dots, u_{d-1}, (\pi^k)_{\mathbf{u}})$ is at most $(v_1, \dots, v_{d-1}, (\pi^k)_{\bs{v}} - 1)$ in the partial order, but is not included in $\mathcal{R}_{\leq k}(\mm)$. This contradicts the fact that $\mathcal{R}_{\leq k}$ is closed downward.
\end{proof}

Hence, for fixed $k$ and $\mm$, the integers $(\pi^k)_\mathbf{v}(\mm)$ form a $(d-1)$-dimensional partition $\pi^k(\mm)$ of $s_1(\mm) + \dots + s_k(\mm)$.

\begin{example}
  On the right side of~\Cref{fig:ex-gns}, we have $\mm = \{(0,2),(2,1),(3,0)\}$ and the regions $\mathcal{R}_1, \mathcal{R}_2, \mathcal{R}_3$ are colored in different shades. We have
  \begin{align*}
    \pi^1(\mm) &= [2,2,1], \\
    \pi^2(\mm) &= [4,4,3,2,2,1], \\
    \pi^3(\mm) &= [6,6,5,4,4,3,2,2,1],
  \end{align*}
  so $(s_1, s_2, s_3) = (5,11,17)$. Also depicted are the sets $\mm$ (denoted by $\bigstar$), $2\mm$ (denoted by $\bullet$), $3\mm$ (denoted by $\triangle$), and $4\mm$ (denoted by $\square$).
\end{example}

In fact, we can define the partitions inductively with the min-sum operation (cf.~\S\ref{sec:prelims}).

\begin{proposition}
  \label{prop:k-sum}
  We have $\pi^k(\mm) = \underbrace{\pi^1(\mm) \boxplus \dots \boxplus \pi^1(\mm)}_{k\text{ times}}$.
  
\end{proposition}

\begin{proof}
  Let $\tau := \pi^1 \boxplus \dots \boxplus \pi^1$ be $\pi^1$ min-summed $k$ times, and define
  \[\mm' = \mm \cup \{(v_1, \dots, v_{d-1}, (\pi^1)_\bs{v}) : \bs{v} \in \NN_0^{d-1} \}.\] Consider the set of points that are not greater than or equal to some element of $k\mm$. This is the same as the set of points that are not greater than or equal to some element of $k\mm'$. Moreover, the minimal elements of $k\mm'$ with respect to $\leq$ are precisely $(v_1, \dots, v_{d-1}, \tau_\bs{v})$ for each $\bs{v}$. Hence, $(\pi^k)_\bs{v} = \tau_\bs{v}$ and so $\pi^k \simeq \tau$.
\end{proof}

Thus, the partition $\pi^1(\mm)$ can be used to compute the other partitions $\pi^k(\mm)$, which characterize $\mathcal{R}_{\leq k}(\mm)$ and the depth-$k$ regions.

\begin{definition}
  The \vocab{shape} of $\Lambda^d$ is the $(d-1)$-dimensional partition $\pi^1( \mm(\Lambda^d) )$, whose total is the \vocab{shape size} $s(\Lambda^d) := s_1(\mm(\Lambda^d))$.
\end{definition}

Note that for $d=1$, the shape size $s(\Lambda)$ is exactly $m(\Lambda)-1$, which happens to be the length of the Kunz word of $\Lambda$ (cf.~\S\ref{sec:part}).

\section{Improved asymptotic bounds on $n_{g,d}$}

\label{sec:bounds}

Let $n_{g,d}$ denote the number of $d$-dimensional GNSs with genus $g$. In this section, we compute some bounds on $n_{g,d}$. 

\subsection{Lower bound}

In this subsection, we show~\Cref{thm:gns-record}. Specifically, we construct a family of GNSs of depth 2 whose size has the desired growth rate. To this end, let $n_{g,d,q}$ denote the number of generalized numerical semigroups with genus $g$, dimension $d$, and depth $q$. The following result implies~\Cref{thm:gns-record}. Surprisingly, the growth factor $\rr_{2^d}$ is sharp for depth 2 GNSs in the following sense.

\begin{theorem}
  \label{thm:ngd2}

  For fixed $d$, there are constants $C_{d}^-, C_{d}^+ > 0$ for which
  \[(C_d^-)^{g^{(d-1)/d}} \rr_{2^d}^g < n_{g,d,2} < (C_d^+)^{g^{(d-1)/d}} \rr_{2^d}^g.\]
\end{theorem}

First, we give an exact formula for $n_{g,d,2}$ as a sum over the possible multsets $\mm$.

\begin{proposition}
  \label{prop:ngd2-form}
  We have
  \[n_{g,d,2} = \sum_{\mm} \binom{s_2(\mm) - \# \mm}{ g + 1 - s_1(\mm)},\]
  where the sum is over all multsets $\mm$ for which $s_1(\mm) \leq g+1$. 
\end{proposition}

\begin{proof}
  Say $\Lambda^d$ has depth 2 and multset $\mm$, so $\mathcal{R}_k \subset \Lambda^d$ for $k \geq 3$. The inclusion of each element of $\mathcal{R}_1$ is predetermined; $\bs{0}$ is included and all other elements are excluded. Moreover, every element of $\mm \subset \mathcal{R}_2$ is included in $\Lambda^d$ by definition. 

  Thus, it remains to choose $g+1-s_1$ elements from $\mathcal{R}_2 \setminus \mm$. We claim that any subset yields a valid $\Lambda^d$. Indeed, by~\Cref{prop:dumb-stuff} we have that if $\bs{x}, \bs{y} \in \mathcal{R}_2$, then $\bs{x} + \bs{y} \not\in \mathcal{R}_{\leq 2}$, so in fact $\mathcal{R}_2$ is sum-free. Hence, we have that $\Lambda^d$ is always closed under addition, and we have the desired claim. 
\end{proof}

It turns out we can bound $s_2$ linearly in terms of $s_1$, which in turns gives us an exponential upper bound in terms of $\rr_{2^d}$. To bound the error term, we use the following formulation of the \vocab{Loomis-Whitney inequality}, proved in 1949, which allows us to bound the sum of the entries of a $(d-1)$-dimensional partition along the coordinate hyperplanes.

\begin{theorem}[Loomis-Whitney {\cite[Thm.~2]{lw_1949}}]
  \label{lem:lw}
  For any set of points $S$ in $d$-space, let $S_i$ be the set of points obtained by projecting $S$ onto the $i$-th coordinate hyperplane. Then
  \[(\# S)^{d-1} \leq \prod_{i=1}^d (\# S_i).\]
\end{theorem}

\begin{lemma}
  \label{lem:r2-bound}
  For any multset $\mm$ and $d \geq 2$, we have $s_2(\mm) \leq (2^d-1)s_1(\mm) - 2^{d-2} s_1(\mm)^{(d-1)/d}$. 
\end{lemma}

\begin{proof}
  We work with the $(d-1)$-dimensional partitions $\pi^1(\mm)$ and $\pi^2(\mm)$ of $s_1$ and $s_1+s_2$, respectively. For each point $\bs{v} \in \NN_0^{d-1}$ and subset $S \subseteq [d-1]$, let $f_S(\bs{v}) := (f_{S,1}(v_1), \dots, f_{S,d-1}(v_{d-1}))$, where
  \[f_{S,i}(v) = \begin{cases}
    \lceil v/2 \rceil & \text{if }i \in S, \\
    \lfloor v/2 \rfloor & \text{if }i \not\in S.
  \end{cases}\]
  By definition, we have $f_S(\bs{v}) + f_{[d-1]\setminus S}(\bs{v}) = \bs{v}$. Moreover, \Cref{prop:k-sum} tells us $\pi^2 = \pi^1 \boxplus \pi^1$, and so
  \begin{align*}
    (\pi^2)_{\bs v} &= \min_{\bs u + \bs{u'} = \bs v} \left( (\pi^1)_{\bs u} + (\pi^1)_{\bs u'} \right) \\
    &\leq \frac{1}{2^{d-1}} \sum_{S \subseteq [d-1]} \left( (\pi^1)_{f_S(\bs{v})} + (\pi^1)_{f_{[d]\setminus S}(\bs{v})} \right) \\
    &= \frac{1}{2^{d-2}} \sum_{S \subseteq [d-1]} (\pi^1)_{f_S(\bs{v})}; \\ \\
    s_1+s_2 &= \sum_{\bs v \in \NN_0^{d-1}} (\pi^2)_{\bs v} \\
    &\leq \sum_{\bs v \in \NN_0^{d-1}} \frac{1}{2^{d-2}} \sum_{S \subseteq [d-1]} (\pi^1)_{f_S(\bs{v})}.
\end{align*}
For each $i$, the $i$-th entry of $f_S(\bs{v})$ is determined by the choice of $v_i$ and whether or not $i$ is in $S$. If we fix the $i$-th entry to be $w_i$, there are exactly $3$ such choices that make $w_i = 0$ and $4$ choices if $w_i \geq 1$. Hence, for each $\bs{w}$ there are $4^{d-1}$ choices of $(\bs{v},S)$ for which $f_{S,i}(\bs{v}) = \bs{w}$ if $\bs{w} \geq (1,1,\dots,1)$ (i.e., if all entries are nonzero) and at most $3 \cdot 4^{d-2}$ otherwise. In particular,
\begin{align*}
s_1 + s_2 &\leq \frac{1}{2^{d-2}} \left( 4^{d-1} \sum_{\bs{w} \in \NN_0^{d-1}} (\pi^1)_{\bs{w}} - 4^{d-2} \sum_{\substack{\bs{w} \in \NN_0^{d-1} \\ \bs{w} \not\geq (1,1,\dots,1)}} (\pi^1)_{\bs{w}}\right)\\
&= 2^ds_1 - 2^{d-2} \cdot \sum_{\substack{\bs{w} \in \NN_0^{d-1} \\ \bs{w} \not\geq (1,1,\dots,1)}} (\pi^1)_{\bs{w}}.
\end{align*}
The summation is equal to the number of points in $\mathcal{R}_1$ that are on the $i$-th coordinate hyperplane for some $i < d$. Without loss of generality, say this is maximal for any ordering of the $d$ coordinate axes. Then by \Cref{lem:lw}, the summation is at least $\max_i (\mathcal{R}_1)_i \geq s_1^{(d-1)/d}$, so $s_1 + s_2$ is at most $2^d s_1 - 2^{d-2} s_1^{(d-1)/d}$, as desired.
\end{proof}

\begin{remark*}
  The choice of $\mm$ in the proof of~\Cref{lem:lower-bound} shows that $(2^d-1)s_1 - O\left( s_1^{(d-1)/d} \right)$ is the best possible bound. This multset is provably best for $d=2$; this is the \href{https://artofproblemsolving.com/community/c6h2976989p26685437}{{\color{blue}third problem}} of the Team Selection Test for the 2023 United States International Math Olympiad team, posed by the author. 
\end{remark*}

\begin{corollary}
  \label{cor:upper}
  For each $d$, there is a constant $C^+_d > 0$ for which $n_{g,d,2} < (C^+_d)^{g^{(d-1)/d}} \rr_{2^d}^{g}$.
\end{corollary}

\begin{proof}
  Since $\pi^1$ is a $(d-1)$-dimensional partition, there are exactly $p_{d-1}(s)$ multsets $\mm$ with $s_1(\mm) = s$. Using~\Cref{lem:r2-bound}, we have
  \begin{align*}
    n_{g,d,2} = \sum_{s=1}^{g+1} \sum_{\substack{\mm \text{ multset} \\ s_1(\mm) = s}} \binom{ s_2(\mm) - \# \mm}{g+1 - s} 
    \leq \sum_{s=1}^{g+1} \sum_{\substack{\mm \text{ multset} \\ s_1(\mm) = s}} \binom{(2^d-1)s}{g+1 - s}
   \leq p_{d-1}(g+1) \cdot \sum_{s=1}^{g+1} \binom{(2^d-1)s}{g+1-s}.
  \end{align*}

  By~\Cref{thm:abp}, there is a constant $K$ for which $p_{d-1}(g+1) < K^{g^{(d-1)/d}}$. On the other hand, the summation is the coefficient of $x^{g+1}$ in the generating function
  \[\sum_{s=1}^{\infty} x^s(x+1)^{(2^d-1)s} = \frac{1}{1-x(x+1)^{2^d-1}}.\]
  The coefficients of the generating function follow a linear recurrence with characteristic polynomial $x^{2^d} - (x+1)^{2^d-1}$, and thus have growth $O(\rr_{2^d}^g)$. The result follows.
\end{proof}

It turns out that the constant $2^d-1$ in~\Cref{lem:r2-bound} is sharp, which suggests that we can choose a specific $\mm$ to give a sufficient lower bound on $n_{g,d,2}$. The proof gives us intuition for the near-equality cases: we should have $\pi^1_{\bs{u}} + \pi^1_{\bs{v}}$ be roughly constant for fixed $\bs{u} + \bs{v}$. We use a specific near-equality case to show the lower bound.

First, we need the following analytic lemmas.

\begin{definition}
  For a positive integer $k$, define the rational function
  \[F_k(x) := \frac{k^k(1-x) x^k}{ ( (k+1)x - 1)^{k+1}}\]
  and let $\cc_k$ be the largest positive root of $F_k(x) - 1 = 0$.
\end{definition}

\begin{lemma}
  The root $\cc_k$ is the unique real root of $F_k(x) -1$ larger than $1/(k+1)$.
\end{lemma}

\begin{proof}
  Let $f(x) = k^k(1-x)x^k$ and $g(x) = ( (k+1)x-1)^{k+1}$. The graph of $f$ has critical points at $x = 0, k/(k+1)$, and is strictly increasing on the interval $(0,k/(k+1))$ and decreasing on the interval $(k/(k+1),\infty)$. It is concave down for $x > (k-2)/k$.

  Meanwhile, the graph of $g$ has a single critical point at $x = 1/(k+1)$, and strictly increases thereafter; moreover, it is concave up.

  First, check that $f(1/(k+1)) > 0 = g(1/(k+1))$ and
  \[g\left( \frac{k}{k+1} \right) = (k-1)^{k+1} > \frac{k^{2k}}{(k+1)^k} = f\left( \frac{k}{k+1} \right).\]
  Since $f$ and $g$ are both increasing on the interval $(1/(k+1),k/(k+1))$, one concave up and the other concave down, the unique value of $x$ for which $f(x) = g(x)$ on this interval is $\cc_k$. Moreover, $g > f$ on $[k/(k+1), \infty)$, since $f$ is decreasing on that interval. We thus have the desired claim.
\end{proof}

\begin{definition}
  For positive integers $g$ and $x \leq g + 1$, define $G_k(x,g) := \binom{kx}{g+1-x}$. For fixed $g$, let $x = r_k(g)$ be the value of $x$ which maximizes $G_k(x,g)$.
\end{definition}

\begin{lemma}
  \label{lem:max-egf}
  We have $G_k(r_k(g),g) \geq K/g \cdot \rr_{k+1}^g$ for some constant $K > 0$.
\end{lemma}

\begin{proof}
  By a standard characteristic polynomial argument as in~\Cref{cor:upper}, we have that \[\sum_{s=1}^{g+1} G_k(s,g) = O(\rr_{k+1}^g).\] But the left-hand side is at most $g \cdot G_k(r_k(g),g)$, which yields the desired bound.
\end{proof}

\begin{lemma}
  \label{lem:gk-converge}
Let $s(1), s(2), \dots$ be a sequence of integers with $s(g) \leq g + 1$. 
  As $g$ goes to infinity,
  \[\frac{G_k(s(g)+1,g)}{G_k(s(g),g)} \sim F_k(s(g)/g).\]
\end{lemma}

\begin{proof}
  Abbreviate $s := s(g)$. Rewrite
  \begin{align*}
    \frac{G_k(s+1,g)}{G_k(s,g)} = \frac{(g+1-s)\prod_{i=1}^k (ks+i)}{\prod_{i=0}^k ( (k+1)s - g + i)} = \frac{(1-s/g+1/g)\prod_{i=1}^k (ks/g+i/g)}{\prod_{i=0}^k ( (k+1)s/g - 1 + i/g)}.
  \end{align*}
  Then the right-hand side tends to $F_k(s/g)$, as desired.
\end{proof}

\begin{corollary}
  \label{lem:bin-const}
  As $g$ approaches infinity, the ratio $r_k(g)/g$ approaches $\cc_k$.
\end{corollary}

\begin{proof}
  Abbreviate $r := r_k(g)$. Our choices of $r$ dictate that
  \[\frac{G_k(r+1,g)}{G_k(r,g)} < 1 \quad \text{and} \quad \frac{G_k(r,g)}{G_k(r-1,g)} > 1.\]

  For large $g$, the two quantities both approach $F_k(r/g)$.
  Hence, by the squeeze theorem, we have that $F_k(r/g)$ approaches $1$, so $r/g$ approaches some root of $F_k$. We must have $kr > g-r$ or $r/g > 1/(k+1)$, so the only possibility is for $r/g$ to approach $\cc_k$, as desired.
\end{proof}

\begin{lemma}
  \label{lem:lower-bound}
  For each positive integer $d$, there exists $C^-_d > 0$ for which $n_{g,d,2} > (C^-_d)^{-g^{(d-1)/d}} \rr_{2^d}^g$.
\end{lemma}

\begin{proof}
  Let $\mm_k$ denote the lattice points on the plane $x_1 + \dots + x_d = k$. Then $\mm_k$ is an antichain with points on the axes and is thus a valid multset, with
  \begin{align*}
    \mathcal{R}_1(\mm_k) &= \{ (x_1, \dots, x_d) : x_1 + \dots + x_d < k \}, \\
    \mathcal{R}_2(\mm_k) \setminus \mm_k &= \{ (x_1, \dots, x_d) : k < x_1 + \dots + x_d < 2k\},
  \end{align*}
  whose cardinalities are
  \[
  \begin{aligned}
    s_{1,k,d} &:= \binom{k+d-1}{d} = \frac{1}{d!} k^d + O(k^{d-1}), \\
    s_{2,k,d} &:= \binom{2k+d-1}{d} - \binom{k+d}{d} = \frac{1}{d!}(2^d-1) k^d + O(k^{d-1}).
  \end{aligned}
  \]
  Thus, there are $t^k_{g,d} := \binom{s_{2,k,d}}{g+1-s_{1,k,d}}$ GNSs with depth 2, genus $g$, and multset $\mm_k$. It suffices to show we can pick a $k$ for which $t^k_{g,d}$ exceeds the desired bound.

  Let $r := r_{2^d-1}(g)$. Since $s_{1,k,d}$ is a polynomial of degree $d$ in $k$, there is a constant $A > 0$ for which we can always select $k$ such that $0 < s_{1,k,d} - r < A \cdot k^{d-1}$. The idea is that we can approximate $t^k_{g,d}$ closely by $G_{2^d-1}(s_{1,k,d}, g)$, which in turn is close to $G_{2^d-1}(r,g)$, the last of which is large.

  Let $\eps > 0$ be a threshold which we later send to $0$. We first compare the quantities $t^k_{g,d}$ and $G_{2^d-1} (s_{1,k,d}, g)$. Let $D = (2^d-1)s_{1,k,d} - s_{2,k,d}$. Since $D$ is a polynomial of degree $d-1$ in $k$ with positive leading coefficient, we know $0 < D < B \cdot k^{d-1}$ for some $B > 0$. Note that
  \[\frac{t^k_{g,d}}{G_{2^d-1}(s_{1,k,d},g)} = \prod_{j = 1}^{D} \frac{2^d s_{1,k,d} - g - j}{(2^d-1)s_{1,k,d} + 1 -j} = \prod_{j=1}^D \frac{2^ds_{1,k,d}/g - 1 - j/g}{(2^d-1)s_{1,k,d}/g + 1/g - j/g}.\]
  Since $s_{1,k,d} - r \ll k^d$, we have that $s_{1,k,d}/g \sim r/g$ which tends to $\cc_{2^d-1}$ by~\Cref{lem:bin-const}. Hence, the right-hand side is bounded below by
  \[\left( \frac{2^d \cc_{2^d-1} - 1}{(2^d-1) \cc_{2^d-1}} - \eps \right)^D > \left( \frac{2^d \cc_{2^d-1}-1}{(2^d-1) \cc_{2^d-1}}-\eps\right)^{B \cdot k^{d-1}}\]
  for sufficiently large $g$.

  Now, we compare $G_{2^d-1}(s_{1,k,d},g)$ and $G_{2^d-1}(r,g)$. Note that
  \[\frac{G_{2^d-1}(s_{1,k,d},g)}{G_{2^d-1}(r,g)} = \prod_{x=r}^{s_{1,k,d}-1} \frac{G_{2^d-1}(x+1,g)}{G_{2^d-1}(x,g)}.\]
  Every factor in the product converges uniformly to $1$ by~\Cref{lem:gk-converge}, so in particular the right-hand side is at least $(1-\eps)^{s_{1,k,d}-r} > (1-\eps)^{A \cdot k^{d-1}}$ for sufficiently large $g$.

  In summary, we have
  \[t^k_{g,d} > \left( \frac{2^d \cc_{2^d-1} - 1}{(2^d-1)\cc_{2^d-1}} - \eps \right)^{B \cdot k^{d-1}} (1-\eps)^{A \cdot k^{d-1}} G_{2^d - 1} (r,g).\]
  But we have $G_{2^d-1}(r,g) > K/g \cdot \rr_{2^d-1}^g$ by~\Cref{lem:max-egf}, and moreover $k = O(g^{1/d})$ since $r = O(g)$ by~\Cref{lem:bin-const}. The result follows.
\end{proof}

Note that the proof can be modified slightly to produce an explicit value of $C_d$, but we do not do that here. \Cref{cor:upper} and~\Cref{lem:lower-bound} collectively imply~\Cref{thm:ngd2}, which in turn implies~\Cref{thm:gns-record}.

\subsection{Upper bound}

\label{sec:upper}

In this section, we provide the first upper bound on $n_{g,d}$ by proving~\Cref{thm:bad-upper}. In 2007, Bras-Amor\'os and de Mier~\cite{amoros_2007} gave the first upper bound on the number $n_g$ of numerical semigroups $\Lambda$ with genus $g$. They noted that if $n \not\in \Lambda$, then at least one of $(k, n-k)$ must also be excluded from $\Lambda$ for each $k$, and thus $g \geq (n+1)/2$. In particular, $f \not\in \Lambda$, so $f \leq 2g-1$. 

We use similar logic to develop an upper bound.

\begin{definition}
  For each $g$, let $\mathcal{A}_g := \{(x_1, \dots, x_d) \in \NN_0^d : \prod_{i=1}^d (x_i+1) \leq 2g\}$.
\end{definition}

\begin{proposition}
  \label{prop:up-bound}
  If $\Lambda^d$ is a GNS with dimension $d$ and genus $g$, then $\NN_0^d \setminus \Lambda^d \subseteq \mathcal{A}_g$.
\end{proposition}

\begin{proof}
  Suppose $\bs{x} \not\in \Lambda^d$. Then for each $\bs{a} \leq \bs{x}$, one of $(\bs{a}, \bs{x}-\bs{a})$ must be also excluded from $\Lambda^d$. In particular, we must exclude at least half of the elements in the prism of points bounded by $\bs{x}$, which contains $\prod_{i=1}^{d}(x_i+1)$ elements. Hence, $g \geq \frac12 \prod_{i=1}^d(x_i+1)$, as desired.
\end{proof}

We now establish our upper bound.

\begin{proof}[Proof of~\Cref{thm:bad-upper}]
   Note that we have
  \[\#\mathcal{A}_g = \sum_{0\leq x_1, \dots, x_{d-1} < 2g} \left\lfloor \frac{2g}{\prod_{i=1}^{d-1} (x_i+1)} \right\rfloor \leq \sum_{0\leq x_1, \dots, x_{d-1} < 2g} \frac{2g}{\prod_{i=1}^{d-1} (x_i+1)} = 2gH_{2g}^{d-1}.\]
  Hence, if $\Lambda^d$ has genus $g$, we choose to exclude $g$ elements from $\mathcal{A}_g$, and thus
  \begin{align*}
    n_{g,d} \leq \binom{2g H_{2g}^{d-1}}{g} \leq \frac{1}{g!} (2g H_{2g}^{d-1})^g.
  \end{align*}
  Stirling's approximation tells us $g! > \left( (1/e - o(1)) g \right)^g$ for sufficiently large $g$, while we have the bound $H_{2g} \leq \ln (2g) + 1$ on the harmonic sum. The result follows. 
\end{proof}

\section{Partition labelings}

\label{sec:part}

In this section, we define \vocab{partition labelings} that somewhat generalize the notion of Kunz words to generalized numerical semigroups.

Given a numerical semigroup $\Lambda$ with multiplicity $m$, recall the \vocab{Kunz word} $\mathcal{K}(\Lambda)$ is the word $w_1 \cdots w_{m-1}$, where $m \cdot w_i + i$ is the least element in $\Lambda$ congruent to $i \pmod{m}$. One can easily recover $\Lambda$ from $\mathcal{K}(\Lambda)$, and we can read off many invariants of $\Lambda$, such as its genus and depth, from its Kunz word. 

We say a word $w_1 \cdots w_{m-1}$ is a \emph{valid Kunz word} if it is the Kunz word of some numerical semigroup. In 1987, Kunz showed the valid Kunz words are exactly those which follow certain additive inequalities known as the \vocab{Kunz conditions}:

\begin{proposition}[{Kunz~\cite[$\S2$]{kunz_1987,rosales_2009}}]
  A word $w_1 \cdots w_{m-1}$ is a valid Kunz word if and only if:
  \begin{itemize}
   \item $w_i + w_j \geq w_{i+j}$ for all $i,j$ with $i+j < m$;
    \item $w_i + w_j + 1 \geq w_{i+j-m}$ for all $i,j$ with $i+j \geq m+1$.
  \end{itemize}
\end{proposition}

These properties make Kunz words a powerful tool to study numerical semigroups. For more on Kunz words, see~\cite{li_2022, zhu_2022}.

In the setting of generalized numerical semigroups, we generalize Kunz words to an object that we call a \vocab{partition labeling}, formed by labeling the elements of $\mathcal{R}_1$ with certain $(d-1)$-dimensional partitions. To begin, we make the following definition.

\begin{definition}
  For a multset $\mm$ and index $i=1,\dots,d$, let $m_i(\mm)$ be the integer $m_i$ for which $m_i \cdot \mathbf{e^d_i} \in \mm$. The \vocab{volume} $V(\mm)$ of a multset is the quantity $\prod_{i=1}^d m_i$.
\end{definition}

Recall from~\Cref{prop:mult-char} that such an $m_i$ exists and is unique, so our definition is sound. 

\begin{definition}
  Given a GNS $\Lambda^d$, the \vocab{partition labeling} $\mathscr{L}(\Lambda^d)$ of $\Lambda^d$ is composed of the points $\bs{x} \leq (m_1-1, \dots, m_d-1)$, each labeled with a partition $L^\bs{x}$ as follows:
  \begin{itemize}
    \item if $\bs{x} \in \Lambda^d$, then $\bs{x}$ is labeled with the empty (zero) partition;
    \item if $\bs{x} \not\in \Lambda^d$, then $\bs{x}$ is labeled with the partition $L^{\bs{x}}$, where
      \[(L^{\bs{x}})_{\bs{v}} = \min \{ \ell : \bs{x} + (m_1v_1, \dots, m_{d-1}v_{d-1}, m_d \ell) \in \Lambda^d\}.\]
  \end{itemize}
\end{definition}

We elucidate the above definition with an example.

\begin{example}
  As in~\Cref{ex:gns-2}, let
  \[\Lambda^2 = \NN_0^2 \setminus \{(0,1),(0,3),(1,0),(1,1),(1,2),(2,0), (3,1), (3,3), (4,1)\}.\]
  Then we have $m_1 = 3$ and $m_2 = 2$.

  We first compute the entries of $L^{(0,1)}$. To compute $\left(L^{(0,1)}\right)_0$, we wish to find the least $\ell$ for which
  \[(0,1) + (3 \cdot 0, 2 \cdot \ell) \in \Lambda^2.\]
  This is false for $\ell=0,1$, but true for $\ell = 2$. Hence, $\left( L^{(0,1)} \right)_0 = 2$.

  Similarly, to compute $\left( L^{(0,1)} \right)_1$, we wish to find the least $\ell$ for which
  \[(0,1) + (3 \cdot 1, 2 \cdot \ell) \in \Lambda^2.\]
  Once again, this is false for $\ell=0,1$ but true for $\ell = 2$. Thus, $\left( L^{(0,1)} \right)_1 = 2$ as well.

  Continuing in this fashion, we get that $\mathscr{L}(\Lambda^d)$ is as follows:
  \begin{align*}
    L^{(0,1)} &= [2,2], \qquad L^{(1,1)} = [1,1], \qquad L^{(2,1)} = [\phantom{1,1}] \\
    L^{(0,0)} &= [\phantom{2,2}], \qquad L^{(1,0)} = [\phantom{1}2\phantom{1}], \qquad L^{(2,0)} = [\phantom{1}1\phantom{1}].
  \end{align*}
  This is depicted (perhaps more intuitively) with Young diagrams in~\Cref{fig:pl}. For instance, to recover $L^{(0,1)}$, we look at the top left corner of each of the bolded $3 \times 2$ boxes on the left-hand side of~\Cref{fig:pl} and see which elements are omitted from $\Lambda^2$.
\end{example}

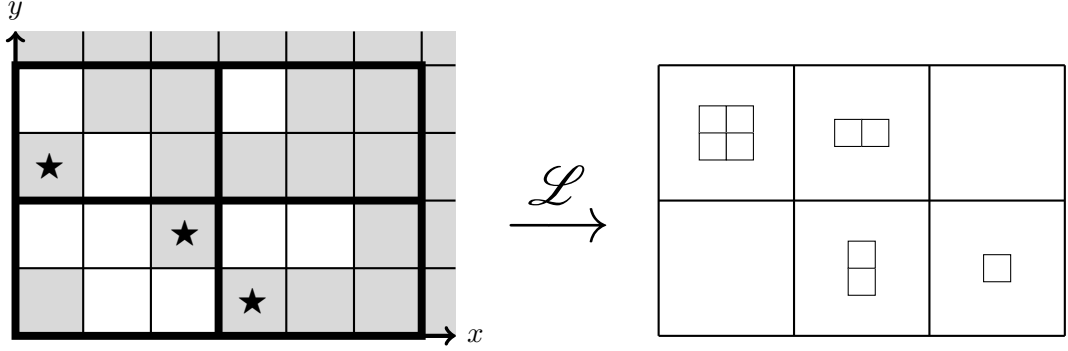
\begin{figure}
  \begin{tikzpicture}[scale=0.9]
    \fill[gray!30!white] (0,0)--(6.5,0)--(6.5,4.5)--(0,4.5)--cycle;
    \foreach \xx/\yy in {0/1,0/3,1/0,1/1,1/2,2/0,3/1,3/3,4/1} {
    \fill[white] (\xx,\yy)--(\xx+1,\yy)--(\xx+1,\yy+1)--(\xx,\yy+1)--cycle;
    }
    \draw[ultra thick,->] (0,0)--(6.5,0) node[right] {$x$};
    \draw[ultra thick,->] (0,0)--(0,4.5) node[above] {$y$};
    
    \foreach \i in {0,...,6} {
    \draw[thick] (\i,0)--(\i,4.5);
    }
    \foreach \i in {0,...,4} {
    \draw[thick] (0,\i)--(6.5,\i);
    }
    \foreach \xx/\yy in {0/2,2/1,3/0} {
    \node at (\xx+0.5,\yy+0.5) {\large $\bigstar$}; 
    }
    \draw[line width=3pt] (0,0)--(6,0)--(6,4)--(0,4)--cycle;
    \draw[ line width=3pt] (3,0)--(3,4);
    \draw[ line width=3pt] (0,2)--(6,2);

    \node at (8,2) {\Huge $\stackrel{\mathscr{L}}{\longrightarrow}$};

    \foreach \i in {0,...,3} {
    \draw[thick] (9.5+2*\i,0)--(9.5+2*\i,4);
    }
    \foreach \i in {0,...,2} {
    \draw[thick] (9.5,2*\i)--(15.5,2*\i);
    }
    \filldraw[color=black, fill=white] (14.3,0.8)--(14.3,1.2)--(14.7,1.2)--(14.7,0.8)--cycle;
    \filldraw[color=black, fill=white] (12.3,0.6)--(12.3,1.0)--(12.7,1.0)--(12.7,0.6)--cycle;
    \filldraw[color=black, fill=white] (12.3,1)--(12.3,1.4)--(12.7,1.4)--(12.7,1)--cycle;
    \filldraw[color=black, fill=white]
    (12.1,2.8)--(12.1,3.2)--(12.5,3.2)--(12.5,2.8)--cycle;
    \filldraw[color=black, fill=white]
    (12.5,2.8)--(12.5,3.2)--(12.9,3.2)--(12.9,2.8)--cycle;

    \filldraw[color=black, fill=white]
    (10.1,3.0)--(10.1,3.4)--(10.5,3.4)--(10.5,3.0)--cycle;
    \filldraw[color=black, fill=white]
    (10.5,3.0)--(10.5,3.4)--(10.9,3.4)--(10.9,3.0)--cycle;

    \filldraw[color=black, fill=white]
    (10.1,3.0-0.4)--(10.1,3.0)--(10.5,3.0)--(10.5,3.0-0.4)--cycle;
    \filldraw[color=black, fill=white]
    (10.5,3.0-0.4)--(10.5,3.0)--(10.9,3.0)--(10.9,3.0-0.4)--cycle;
  \end{tikzpicture}
  \caption{Partition labeling of $\Lambda^2$.}
  \label{fig:pl}
\end{figure}

We first check that the $L^\bs{x}$ are indeed partitions.

\begin{proposition}
  If $\bs{u} \leq \bs{v}$, then $(L^{\bs{x}})_{\bs{u}} \geq (L^\bs{x})_\bs{v}$.
\end{proposition}

\begin{proof}
  Suppose $(L^\bs{x})_{\bs{u}} < (L^{\bs{x}})_{\bs{v}}$. Then we have that $\bs{a} := \bs{x} + (m_1u_1, \dots, m_{d-1}u_{d-1} + m_d(L^{\bs{x}})_{\bs{v}})$ is in $\Lambda^d$, but $\bs{b} := \bs{x} + (m_1v_1, \dots, m_{d-1}v_{d-1} + m_d (L^{\bs{x}})_{\bs{v}})$ is not. However, since $m_i\mathbf{e^d_i}$ is in $\Lambda^d$, we have that $\bs{a} + (m_1(v_1-u_1), \dots, m_{d-1}(v_{d-1}-u_{d-1}), 0) = \bs{b}$ is in $\Lambda^d$, which is a contradiction.
\end{proof}

Note we can easily recover $\Lambda^d$ from $\mathscr{L}(\Lambda^d)$, so $\mathscr{L}$ is a map from GNSs to valid partition labelings. Thus, we naturally investigate which partition labelings give rise to valid GNSs. We first offer the following operation on partitions, which will help us extend the one-dimensional Kunz condition $w_x + w_y + 1 \geq w_{x+y-m}$ to the general case.

\begin{definition}
  Given a $(d-1)$-dimensional partition $\pi$ and a subset $X \subseteq [d]$, let $\operatorname{sh}_X(\pi)$ be the partition $\operatorname{sh}_X$ given by
  \[(\operatorname{sh}_X)_{\bs{v}} = \begin{cases}
    \pi_{\bs{v}(X)} & \text{if } d \not\in X, \\
    \max\left(\pi_{\bs{v}(X)}-1,0\right) & \text{if } d \in X,
  \end{cases}\]
  where $\bs{v}(X) := \bs{v} + \sum_{i \in X \cap [d-1]} \mathbf{e^{d-1}_i}$.
\end{definition}

One can think of $\operatorname{sh}_X$ as a geometric operation on multi-dimensional Young diagrams that ``shaves off'' the blocks along the hyperplanes indexed by $X$.

\begin{example}
  The one-dimensional partition $\pi = [4,3,2,2,1,1]$ has 
  \begin{align*}
    \operatorname{sh}_{\{1\}}(\pi) &= [3,2,2,1,1],\\
    \operatorname{sh}_{\{2\}}(\pi) &= [3,2,1,1],\\
    \operatorname{sh}_{\{1,2\}}(\pi) &= [2,1,1].
  \end{align*}
  \Cref{fig:shave} shows the Young diagrams of these partitions in relation to each other.

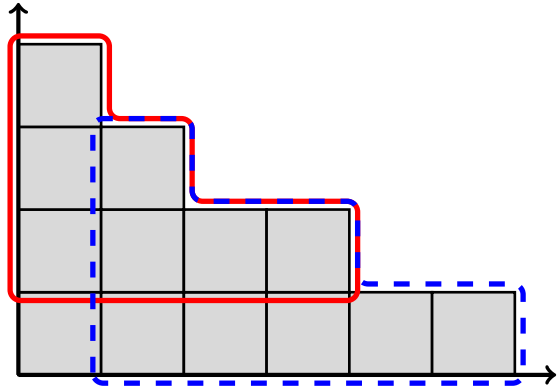
\begin{figure}
\begin{center}
  \begin{tikzpicture}[scale=1.1]
    \foreach \xx/\yy in {0/0,0/1,0/2,0/3,1/0,1/1,1/2,2/0,2/1,3/0,3/1,4/0,5/0} { 
    \filldraw[color=black,fill=gray!30!white,line width=1pt] (\xx,\yy)--(\xx+1,\yy)--(\xx+1,\yy+1)--(\xx,\yy+1)--cycle;
    }
    \draw[->, ultra thick] (0,0)--(0,4.5);
    \draw[->, ultra thick] (0,0)--(6.5,0);
    \draw[rounded corners, line width=2pt, red] (-0.1,0.9)--(-0.1,4.1)--(1.1,4.1)--(1.1,3.1)--(2.1,3.1)--(2.1,2.1)--(4.1,2.1)--(4.1,0.9)--cycle;
    \draw[rounded corners, line width=2pt, blue, dash pattern=on 6pt off 6pt] (0.9,-0.1)--(0.9,3.1)--(2.1,3.1)--(2.1,2.1)--(4.1,2.1)--(4.1,1.1)--(6.1,1.1)--(6.1,-0.1)--cycle;
  \end{tikzpicture}
\end{center}
  \caption{The partitions $\operatorname{sh}_X(\pi)$ for $\pi = [4,3,2,2,1,1]$.}
  \label{fig:shave}
\end{figure}
\end{example}

For two $(d-1)$-dimensional partitions $\pi$ and $\pi'$, we say $\pi \geq \pi'$ if $\pi_\bs{v} \geq \pi'_{\bs{v}}$ for all $\bs{v} \in \NN_0^{d-1}$.

\begin{theorem}
  A partition labeling $\mathscr{L} = \{L^\bs{x} : \bs{x} \leq (m_1-1, \dots, m_d-1)\}$ corresponds to a valid GNS if and only if 
  \begin{itemize}
    \item $L^{\bs{0}} = \varnothing$; and 
    \item for all $\bs{x}, \bs{y}$ we have $L^\bs{x} \boxplus L^\bs{y} \geq \operatorname{sh}_X(L^{\bs{z}})$, where $X$ consists of the indices $j$ for which $x_j + y_j \geq m_j$, and moreover $z_i := (x_i + y_i) \mod{m_i}$.
  \end{itemize}
\end{theorem}

\begin{proof}
  The set of points $\Lambda^d$ corresponding to $\mathscr{L}$ correspond to a valid GNS if and only if the set is closed under addition. Any point in $\Lambda^d$ is expressible in the form $\bs{x} + (m_1v_1, \dots, m_{d-1}v_{d-1}, m_d\ell)$ where $\bs{x} \leq (m_1-1, \dots, m_d-1)$. 

  Now, select $\bs{x} + (m_1v_1, \dots, m_{d-1}v_{d-1}, m_d\ell)$ and $\bs{y} + (m_1w_1, \dots, m_{d-1}w_{d-1}, m_dk)$ from $\Lambda^d$. These must satisfy $\ell \geq (L^{\bs{x}})_{\bs{v}}$ and $k \geq (L^{\bs{y}})_{\bs{w}}$. We require that their sum
  \begin{multline*}
  \bs{x} + \bs{y} + (m_1(v_1+w_1), \dots, m_{d-1}(v_{d-1}+w_{d-1}), m_d(\ell+k)) \\ = \ \bs{z} + \left(\sum_{i \in X} m_i\mathbf{e_i^d}\right) + (m_1(v_1+w_1), \dots, m_{d-1}(v_{d-1} + w_{d-1}), m_d(\ell+k))
\end{multline*}
  is in $\Lambda^d$, too. Let $\bs{a} := (v_1, \dots, v_{d-1},\ell)+ (w_1,\dots,w_{d-1},k) + \sum_{i \in X} \mathbf{e_i^d}$. The above is true if and only if $a_d \geq (L^\bs{z})_{(a_1, \dots, a_{d-1})}$, or
  \begin{itemize}
    \item $\ell + k \geq (L^{\bs{z}})_{(a_1, \dots, a_{d-1})}$ if $d \not \in X$, and
    \item $\ell + k \geq (L^{\bs{z}})_{(a_1, \dots, a_{d-1})}-1$ if $d \in X$.
  \end{itemize}
  Note the right-hand side is exactly $\left(\operatorname{sh}_X(L^{\bs{z}})\right)_{\bs{v} + \bs{w}}$. 

  So essentially, the condition that $\Lambda^d$ is closed under addition is equivalent to the following: if $\ell \geq (L^{\bs{x}})_{\bs{v}}$ and $k \geq (L^{\bs{y}})_{\bs{w}}$, then $\ell + k \geq \left(\operatorname{sh}_X(L^{\bs{z}})\right)_{\bs{v}+\bs{w}}$. Fixing $\bs{u} := \bs{v} + \bs{w}$, we thus have that
  \[\min_{\bs{u} = \bs{v} + \bs{w}} \left( (L^{\bs{x}})_{\bs{v}} + (L^{\bs{y}})_{\bs{w}} \right) \geq \left( \operatorname{sh}_X(L^{\bs{z}}) \right)_{\bs{u}}.\]
  The left-hand side is precisely $\left( L^{\bs{x}} \boxplus L^{\bs{y}} \right)_{\bs{u}}$, so the desired claim follows.
\end{proof}

Evidently, the sum of the entries over all partitions $L^\bs{x} \in \mathscr{L}(\Lambda^d)$ is the genus $g(\Lambda^d)$.

\subsection{GNSs with fixed multset} In this section, we study the number of GNSs of fixed multset, and thus fixed shape. The following result is immediate.

\begin{proposition}
  \label{prop:small}
  The number of GNSs with multset $\{\bs{e_1^d}, \dots, \bs{e_{i-1}^d}, 2\cdot \bs{e_i^d}, \bs{e_{i+1}^d}, \dots, \bs{e_d^d}\}$ and genus $g$ is $p_{d-1}(g)$.
\end{proposition}

\begin{proof}
  Suppose $\Lambda^d$ has genus $g$ and the given multset. Then the partition labeling $\mathscr{L}(\Lambda^d)$ has $L^\bs{0} = \varnothing$ and $L^\bs{e_i^d}$ being some partition of $g$. There are exactly $p_{d-1}(g)$ choices of $L^{\bs{e_i^d}}$, all of which give rise to valid GNSs, so the result follows.
\end{proof}

\begin{corollary}
  The number of $d$-dimensional GNSs with shape size $2$ is $d \cdot p_{d-1}(g)$.
\end{corollary}

\begin{proof}
  If $s(\Lambda^d) = 2$, then $\mm(\Lambda^d)$ has $d$ elements and consists of $2\bs{e_i^d}$ for some $i$, and $\bs{e_{j}^d}$ for $j \neq i$. Then apply~\Cref{prop:small} for each $i$.
\end{proof}

However, it is significantly more difficult to count the number of GNSs with shape size at least $3$. For instance, in the case of $d=2$, the possible shapes of size $3$ are:
\begin{itemize}
  \item $[3]$, which corresponds to $\mm = \{(0,3),(1,0)\}$;
  \item $[2,1]$, which corresponds to $\mm = \{(0,2),(1,1),(2,0)\}$; and
  \item $[1,1,1]$, which corresponds to $\mm = \{(0,1),(3,0)\}$.
\end{itemize}
One can check that
\begin{itemize}
  \item the number $n_{g,2}^{[3]}$ of GNSs with shape $[3]$ is equal to the number of pairs of partitions $\pi, \pi'$, whose total is $g$, such that $\pi \boxplus \pi \geq \pi'$ and $\pi' \boxplus \pi' \geq \operatorname{sh}_{\{1\}}(\pi)$;
  \item the number $n_{g,2}^{[2,1]}$ of GNSs with shape $[2,1]$ is equal to the number of pairs of partitions $\pi,\pi'$, whose total is $g$, such that $\pi \geq \operatorname{sh}_{\{1,2\}} (\pi')$ and $\pi' \geq \operatorname{sh}_{\{1,2\}} (\pi)$; and
  \item we have $n_{g,2}^{[1,1,1]} = n_{g,2}^{[3]}$.
\end{itemize}

We have computed $n_{g,2}^{[3]}$ and $n_{g,2}^{[2,1]}$ for $g \leq 50$, shown in~\Cref{tab:3shape}. These sequences do not yet appear on the OEIS and do not appear to have a well-behaved closed form.

\begin{table}[]
  \small
\begin{tabular}{|rrr|rrr|rrr|rcr|}
\hline
$g$ & $n_{g,2}^{[3]}$ & $n_{g,2}^{[2,1]}$ & $g$ & $n_{g,2}^{[3]}$ & $n_{g,2}^{[2,1]}$ & $g$ & $n_{g,2}^{[3]}$ & $n_{g,2}^{[2,1]}$ & $g$                  & \multicolumn{1}{r}{$n_{g,2}^{[3]}$} & $n_{g,2}^{[2,1]}$     \\ [0.5ex] \hline
1   & 0               & 0                 & 14  & 1028            & 1675              & 27  & 92415           & 105990            & 40                   & \multicolumn{1}{r}{3546174}         & 2908311               \\
2   & 1               & 1                 & 15  & 1526            & 2422              & 28  & 125261          & 139819            & 41                   & \multicolumn{1}{r}{4587402}         & 3671626               \\
3   & 4               & 4                 & 16  & 2241            & 3462              & 29  & 168974          & 183648            & 42                   & \multicolumn{1}{r}{5918389}         & 4623480               \\
4   & 8               & 10                & 17  & 3251            & 4900              & 30  & 227020          & 240224            & 43                   & 7615125                                 & 5807744               \\
5   & 14              & 22                & 18  & 4691            & 6874              & 31  & 303674          & 312984            & 44                   & 9773454                                 & 7277974               \\
6   & 27              & 43                & 19  & 6697            & 9560              & 32  & 404646          & 406255            & 45                   & 12512191                                 & 9099348               \\
7   & 45              & 76                & 20  & 9503            & 13198             & 33  & 537092          & 525424            & 46                   & 15980127                                 & 11351083              \\
8   & 73              & 129               & 21  & 13387           & 18092             & 34  & 710360          & 677201            & 47                   &  20361285                                & 14129340              \\
9   & 118             & 210               & 22  & 18747           & 24636             & 35  & 936150          & 869940            & 48                   &  25885096                                & 17550599              \\
10  & 189             & 331               & 23  & 26074           & 33344             & 36  & 1229632         & 1113989           & 49                   &  32834413                                & 21755722              \\
11  & 293             & 510               & 24  & 36073           & 44873             & 37  & 1609732         & 1422136           & 50                   &  41560508                                & 26914894              \\
12  & 454             & 771               & 25  & 49595           & 60058             & 38  & 2100858         & 1810194           & \multicolumn{1}{l}{} & \multicolumn{1}{l}{}                & \multicolumn{1}{l|}{} \\
13  & 684             & 1144              & 26  & 67874           & 79977             & 39  & 2733427         & 2297616           & \multicolumn{1}{l}{} & \multicolumn{1}{l}{}                & \multicolumn{1}{l|}{} \\ \hline
\end{tabular}
\caption{Values of $n_{g,2}^{[3]}$ and $n_{g,2}^{[2,1]}$ for $g \leq 50$.}
\label{tab:3shape}
\end{table}

Nevertheless, we can write down the following coarse asymptotic bound on the number of GNSs with fixed multset.

\begin{proposition}
  For each $d$, there is a constant $K_d > 0$ for which the number of GNSs with genus $g$ and multset $\mm$ is at most $K_d^{g^{(d-1)/d} V(\mm)^{1/d}}$.
\end{proposition}

\begin{proof}
  By~\Cref{thm:abp}, there is a constant $k^+ > 0$ for which $p_{d-1}(n) < \exp(k^+ n^{(d-1)/d})$ for all $n$. A partition labeling $\mathscr{L}(\Lambda^d)$ consists of at most $V(\mm)$ nonzero partition labels whose sum is equal to $g$. Namely, if we require that $L^{\bs{x}}$ is a partition of size $n^{\bs{x}}$, then there are at most
  \[\prod_{\bs{x} \leq (m_1-1, \dots, m_d-1)} p_{d-1}(n^{\bs{x}}) \leq \exp\left(k^+ \sum_{\bs{x} \leq (m_1-1, \dots, m_d-1)} (n^{\bs{x}})^{(d-1)/d} \right)\]
  choices of $\mathscr{L}(\Lambda^d)$. Since $g = \sum n^{\bs{x}}$ and $f(x) = x^{(d-1)/d}$ is concave, by Karamata's inequality on concave functions we have that the right-hand side is at most
  \[\exp\left( k^+ (g/V)^{(d-1)/d} V \right).\]
  There are at most $\binom{V+g-1}{g} = o\left( \exp\left(g^{(d-1)/d} V^{1/d} \right) \right)$ choices of $n^{\bs{x}}$, so the result follows.
\end{proof}

\subsection{Rectangular GNSs}

We say that a GNS $\Lambda^d$ is \vocab{rectangular} if $\mathcal{M}(\Lambda^d)$ has size $d$. In particular, this means that $\mm(\Lambda^d) = \{m_1\mathbf{e^d_1}, \dots, m_d \mathbf{e_d^d}\}$. We denote this by using the square symbol $\square$ in the exponent. The partition labeling of a rectangular GNS $\Lambda^{d,\square}$ naturally encodes the depth $q$ in the following way.

\begin{proposition}
  The depth of a rectangular GNS $\Lambda^{d,\square}$ is exactly
  \[ \max_{L^{\bs{x}} \in \mathscr{L}(\Lambda^{d,\square})} \max_{\substack{\bs{v} \in \NN_0^d \\ (L^{\bs{x}})_\bs{v} > 0}} ( v_1 + \dots + v_{d-1} + (L^{\bs{x}})_\bs{v}).\]
\end{proposition}

\begin{proof}
  If $\mm = \{m_1\bs{e_1^d}, \dots, m_d\bs{e_d^d}\}$, then we have that
  \[k\mm = \{(a_1m_1, \dots, a_dm_d) : a_1 + \dots + a_d = k\}.\]
  Thus, for each $\bs{x} \leq (m_1-1, \dots, m_d-1)$, the point $\bs{x} + (m_1v_1, \dots, m_{d-1}v_{d-1}, m_d\ell) \not\in \Lambda^d$ is part of $\mathcal{R}_k$ if and only if $v_1 + \dots + v_{d-1} + \ell = k$. The maximal value of the left-hand sum over all choices of $(\bs{v}, \ell)$ is exactly the expression given in the theorem. But by~\Cref{cor:dumb}, the depth $q$ is the greatest integer $k$ for which $\mathcal{R}_k \not\subseteq \Lambda^{d,\square}$, from which the result follows after letting $\bs{x}$ vary.
\end{proof}

Recall that $P_d(x) := \sum_{n \geq 0} p_d(n)x^n$ is the multi-dimensional partition generating function.

\begin{theorem}
  Let $\mathsf{r}'_d$ be the unique positive root of $P_d(1/x) = 2$. Then we have
  \[\mathsf{r}_d \leq \limsup_{g\to\infty} \left( n^{\square}_{g,d} \right)^{1/g} \leq \mathsf{r}'_d.\]
\end{theorem}

\begin{proof}
  We show the lower and upper bounds separately. Let $n_{g,d}^\mm$ be the number of $d$-dimensional GNSs with genus $g$ and multset $\mm$. In each case, we estimate $n_{g,d}^\mm$ for fixed $\mm$ of size $d$, then we sum over all $\mm$.

Suppose $\Lambda^d$ has multset $\mm = \{m_1 \mathbf{e_1^d}, \dots, m_d\mathbf{e_d^d}\}$.

  \paragraph{Lower bound.} We construct a family of valid partition labelings $\mathscr{L}(\Lambda^d)$ that match the lower bound. Let $\tau$ denote the $(d-1)$-dimensional partition given by $\tau_{\bs{0}} = 1$ and $\tau_{\bs{v}} = 0$ for $\bs{v} > \bs{0}$. If every partition in $\mathscr{L}$ is nonzero and less than or equal to $\tau \boxplus \tau$, then evidently $\mathscr{L}$ is a valid partition labeling. 

  There are exactly $2^d$ nonempty partitions less than or equal to $\tau \boxplus \tau$, of which $\binom{d}{s-1}$ have sum $s \geq 1$. Hence, we have that
  \[\sum_{g=1}^\infty n^{\mm}_{g,d} x^g \geq \left( x(x+1)^d \right)^{V(\mm)}\]
  for any $x > 0$. By summing over all $\mm$, we can show that

  \begin{align*}
    \sum_{g=1}^\infty n_{g,d}^\square x^g &= \sum_{\substack{\mm \text{ multset} \\ |\mm| = d}} \sum_{g=1}^\infty n_{g,d}^\mm x^g 
    \geq \sum_{\substack{\mm \text{ multset} \\ |\mm| = d}} \left( x(x+1)^d \right)^{V(\mm)} 
    \geq \sum_{V=1}^{\infty} \left( x (x+1)^d \right)^V.
  \end{align*}
  By the root test, the right-hand side converges if $\limsup_{V \to \infty} x(x+1)^d \leq 1$
  or $x < 1/\rr_d$. But then the left-hand side converges when $x < \left(\limsup_{g \to \infty} \left(n_{g,d}^\square\right)^{1/g}\right)^{-1}$, which forces the limit supremum to be at least $\mathsf{r}_d$.

(In the language of depth-2 regions, cf.~\S\ref{sec:bounds}, we are counting the number of depth-2 rectangular GNSs. However, we choose to work with partition labelings here to mirror the proof of the upper bound below.)

  \paragraph{Upper bound.} If $\Lambda^d$ has multset $\mm = \{m_1\mathbf{e_1^d}, \dots, m_d\mathbf{e_d^d}\}$, then $\mathscr{L}(\Lambda^d)$ consists of a $m_1 \times \cdots \times m_d$ prism whose entries are all labeled with nonzero $(d-1)$-dimensional partitions. In particular, we have that
  \[\sum_{g=1}^\infty n_{g,d}^{\mm}x^g \leq \left( \sum_{k=1}^\infty p_d(k)x^k \right)^{V(\mm)} = (P_d(x) - 1)^{V(\mm)}\]
  for any $x > 0$. By summing over all $\mm$, we get that
  \begin{align*}
    \sum_{g=1}^\infty n_{g,d}^\square x^g &= \sum_{\substack{\mm \text{ multset} \\ |\mm| = d}} \sum_{g=1}^\infty n_{g,d}^\mm x^g 
    \leq \sum_{\substack{\mm \text{ multset} \\ |\mm| = d}} \left( P_d(x) - 1 \right)^{V(\mm)}\
    \leq \sum_{V=1}^{\infty} \sigma_0(V)^d \left( P_d(x) - 1 \right)^V,
  \end{align*}
 since there are at most $\sigma_0(V)^d$ choices of $m_1, m_2, \dots, m_d$ which multiply to $V$. (Here, $\sigma_0(V)$ denotes the number of divisors of $V$.) By the root test, the right-hand side converges only if
  \[\limsup_{V \to \infty} \left( \sigma_0(V)^d \left( P_d(x) - 1 \right)^V \right)^{1/V} = \limsup_{V \to \infty} (P_d(x) - 1) \leq 1,\]
  which is not true if $P_d(x) > 2$. Then we finish in a similar fashion to the lower bound. 
\end{proof}

\section{Future directions}

\label{sec:future}

In this section, we discuss possible lines of future work by sharpening bounds on $n_{g,d}$ and better understanding partition labelings.

\subsection{Sharpening asymptotics}

In~\Cref{sec:bounds}, we show an exponential lower bound and a superexponential upper bound on $n_{g,d}$. It is natural to ask whether we can reconcile these bounds, since it is somewhat unclear whether $n_{g,d}$ grows exponentially or superexponentially.

\begin{question}
  Is the quantity $n_{g,d}^{1/g}$ bounded?
\end{question}

Part of the difficulty of this question is the disparity between the sizes of $s_1(\mm)$ and $V(\mm)$; the excluded elements of a numerical semigroup can be ``skinny'' along each of the coordinate axes, which yields a small shape size but a large volume for $d \geq 2$. For instance, this behavior is exhibited in the set $\mathcal{A}_g$ of possible points excluded from a genus $g$ GNS (cf.~\S\ref{sec:upper}), which has shape size $O(g (\ln g)^{d-1})$ but volume $O(g^d)$. If $n_{g,d}$ does grow exponentially, it is also natural to ask whether our lower bound of $\rr_{2^d}$ is sharp, and if so, whether the subexponential factor is $C_d^{g^{(d-1)/d}}$.

\begin{question}
  Does the limit $\lim_{g\to\infty} \rr_{2^d}^{-g^{1/d}}n_{g,d}^{g^{-(d-1)/d}}$ exist?
\end{question}

A key ingredient of Zhai's proof is the conjecture of Zhao that almost all numerical semigroups have small depth after ordering by genus~\cite[Conj.~2]{zhai_2012}. This is no longer true for the general case, at least in terms of exponential growth and for our definition of depth. Take the following example, which also exhibits the aforementioned ``axial skinniness.''

\begin{example}
  Let $\mm_k$, $s_{1,k,d}$, and $s_{2,k,d}$ be defined as in the proof of~\Cref{lem:lower-bound}. Suppose $\Lambda^d$ has genus $g$ and satisfies the following properties:
  \begin{itemize}
    \item we have $t \cdot \mathbf{e_1^d} \not\in \Lambda^2$ for every positive integer $t \leq (q-1)k + 1$ and $k \nmid t$;
    \item but it contains every other point in a depth-$k$ region for $k \geq 3$.
  \end{itemize}

  One can check that $\Lambda^d$ must be a GNS as follows. Suppose $\bs{a}, \bs{b} \in \Lambda^d$ are nonzero.
  \begin{itemize}
    \item \emph{Case 1}: Both $\bs{a}$ and $\bs{b}$ are multiples of $\mathbf{e_1^d}$. Note that $S = \NN_0 \setminus \{t \in \NN_0 : t \leq (q-1)k + 1, \ k \nmid t\}$ is a numerical semigroup. Moreover, we have $s \cdot \mathbf{e_1^d} \in \Lambda^d$ if and only if $s \in S$. Hence, since $a_1, b_1 \in S$, we have $a_1 + b_1 \in S$, ergo $\bs{a} + \bs{b} \in \Lambda^d$.
    \item \emph{Case 2}: Either $\bs{a}$ or $\bs{b}$ are not multiples of $\mathbf{e_1^d}$. Then $\bs{a} + \bs{b}$ is in a region of at least $3$ by~\Cref{prop:dumb-stuff} and also is not a multiple of $\mathbf{e_1^d}$. This guarantees $\bs{a} + \bs{b} \in \Lambda^d$ by construction.
  \end{itemize}

  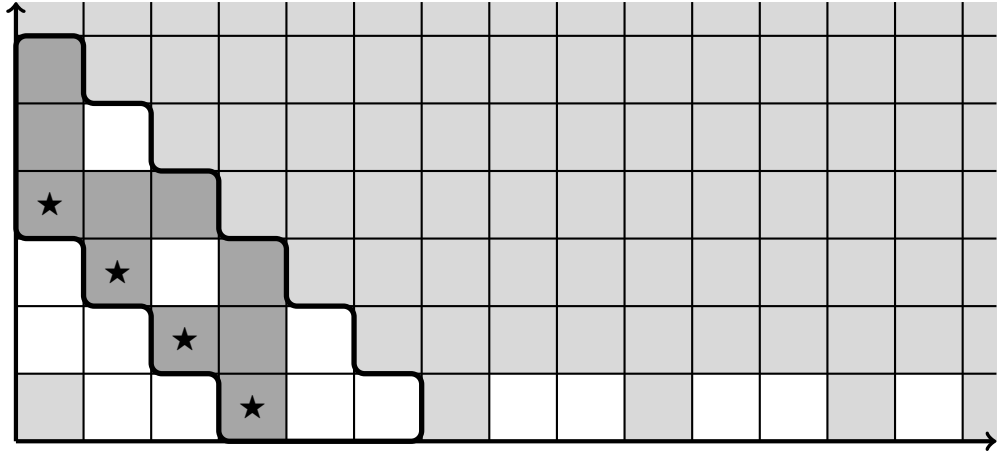
\begin{figure}
    \begin{center}
  \begin{tikzpicture}[scale=0.899]
    \fill[fill=gray!30!white] (0,0)--(14.5,0)--(14.5,6.5)--(0,6.5)--cycle;
    \foreach \xx/\yy in {0/3,0/4,0/5,1/2,1/3,1/4,2/2,2/3,2/1,3/1,3/2,3/0,4/0,4/1,5/0} {\fill[gray!70!white] (\xx,\yy)--(\xx+1,\yy)--(\xx+1,\yy+1)--(\xx,\yy+1)--cycle;}
    \foreach \xx/\yy in {1/0, 2/0, 0/1, 1/1, 0/2, 4/0,5/0,7/0,8/0,10/0,11/0,13/0,2/2,1/4,4/1} {\fill[white] (\xx,\yy)--(\xx+1,\yy)--(\xx+1,\yy+1)--(\xx,\yy+1)--cycle;}
    \foreach \ii in {1,...,14} { \draw[thick] (\ii,0)--(\ii,6.5); }
    \foreach \ii in {1,...,6} { \draw[thick] (0,\ii)--(14.5,\ii); }
    \node at (3.5,0.5) {$\bigstar$};
    \node at (2.5,1.5) {$\bigstar$};
    \node at (1.5,2.5) {$\bigstar$};
    \node at (0.5,3.5) {$\bigstar$};
    \draw[->, ultra thick] (0,0) -- (14.5,0);
    \draw[->, ultra thick] (0,0) -- (0,6.5);
    \draw[line width = 2pt, rounded corners] (0,3)--(0,6)--(1,6)--(1,5)--(2,5)--(2,4)--(3,4)--(3,3)--(4,3)--(4,2)--(5,2)--(5,1)--(6,1)--(6,0)--(3,0)--(3,1)--(2,1)--(2,2)--(1,2)--(1,3)--cycle;
  \end{tikzpicture}
\end{center}
\caption{A GNS with multset $\mm_3$ and depth $5$, with $\mathcal{R}_2(\mm_3)$ outlined.}
\label{fig:counter}
\end{figure}

  Then there are $g-s_{1,k,d}-(q-2)(k-1)$ more points to exclude from $\mathcal{R}_2(\mm_k)$, which has $s_{2,k,d}-k+1$ remaining elements.
Hence, there are $\binom{s_{2,k,d}-k+1}{g-s_{1,k,d}-(q-2)(k-1)}$ GNSs that satisfy the given requirements. By holding $q$ constant and letting $g,k$ grow, this quantity has exponential growth factor $\rr_{2^d}$. However, all of these GNSs have depth $q$. For an example of a GNS with $k=3$ and $q=5$, see~\Cref{fig:counter}. 
\end{example}

However, the data from small cases suggests that most GNSs are depth $2$ or $3$, akin to the one-dimensional case. A natural guess for $d \geq 2$ would be that all $n_{g,d,q}$ have exponential growth rate $\rr_{2^d}$, but the subexponential growth factor $\rr_{2^d}^{-g^{1/d}} n_{g,d,q}^{g^{-(d-1)/d}}$ is largest for $q=3$. In this paper, we describe a large class of depth $2$ GNSs; it would be interesting to examine large classes of depth $3$ GNSs, analogous to the resuls of Zhao for $d=1$~\cite{zhao_2011}. 

\begin{question}
  Do $d$-dimensional GNSs almost all have depth $3$ for $d \geq 2$?
\end{question}

We have explicitly calculated $n_{g,2,q}$ for the small cases of $g \leq 14$ and $q \leq 5$, shown in~\Cref{tab:ng2q}, whose numerics support our conjecture. It is not difficult to show that $n_{g,d,1} = p_d(g+1)$ and $n_{g,d,g} = d^2$ for $g \geq 2$. However, the columns of the table are not yet on the OEIS, so it would be interesting to see if these sequences have other combinatorial significance.

\begin{table}[]

\small

\begin{tabular}{|c|rrrrr|}
\hline
$g$ & $n_{g,2,1}$ & $n_{g,2,2}$ & $n_{g,2,3}$ & $n_{g,2,4}$ & $n_{g,2,5}$ \\ \hline
1   & 2           & 0           & 0           & 0           & 0           \\
2   & 3           & 4           & 0           & 0           & 0           \\
3   & 5           & 14          & 4           & 0           & 0           \\
4   & 7           & 48          & 12          & 4           & 0           \\
5   & 11          & 143         & 44          & 8           & 4           \\
6   & 15          & 412         & 163         & 36          & 8           \\
7   & 22          & 1176        & 550         & 106         & 28          \\
8   & 30          & 3332        & 1751        & 333         & 86          \\
9   & 42          & 9287        & 5514        & 1009        & 254         \\
10  & 56          & 25630       & 17080       & 3065        & 737         \\
11  & 77          & 70574       & 52028       & 9128        & 2133        \\
12  & 101         & 194290      & 156358      & 26985       & 6053        \\
13  & 135         & 534127      & 465726      & 78983       & 16992       \\
14  & 176         & 1465245     & 1377185     & 228727      & 47225       \\ \hline
\end{tabular}

\medskip

\caption{Values of $n_{g,2,q}$ for $g \leq 14$ and $q \leq 5$.}
\label{tab:ng2q}
\end{table}

\Cref{thm:ngd2} implies that the quantity $\rr_{2^d}^{-g^{1/d}} n_{g,d,2}^{g^{-(d-1)/d}}$ for depth $2$ GNSs is bounded. By implementing~\Cref{prop:ngd2-form}, we have calculated the values of $n_{g,2,2}$ for $n \leq 87$, shown in \Cref{tab:ng22}. The values of $\rr_4^{-\sqrt{g}} n_{g,2,2}^{1/\sqrt{g}}$ are graphed against $1/g$ in \Cref{fig:plot}, which suggests that this quantity converges to a constant near $1.2$. 

\begin{table}[]
  \footnotesize
\begin{tabular}{|rr|rr|rr|}
\hline
$g$ & $n_{g,2,2}$   & $g$ & $n_{g,2,2}$                & $g$ & $n_{g,2,2}$                            \\ \hline
0   & 0             & 30  & 12459909670309             & 60  & 80985430675574711412980916             \\
1   & 0             & 31  & 33519288444409             & 61  & 215509768877495232586787465            \\
2   & 4             & 32  & 90136456224494             & 62  & 573383202769145098057975309            \\
3   & 14            & 33  & 242283690207403            & 63  & 1525226173996843571825323845           \\
4   & 48            & 34  & 650936600796631            & 64  & 4056291288735430727151164447           \\
5   & 143           & 35  & 1747891377256538           & 65  & 10785145844024419432004114254          \\
6   & 412           & 36  & 4690642296534889           & 66  & 28669745119349640035022238173          \\
7   & 1176          & 37  & 12580211126984860          & 67  & 76194552501074658365836459077          \\
8   & 3332          & 38  & 33720107313956188          & 68  & 202455550832885616509159776241         \\
9   & 9287          & 39  & 90333780254836434          & 69  & 537831233961624997213173542362         \\
10  & 25630         & 40  & 241874514915972126         & 70  & 1428495078136679841557819365161        \\
11  & 70574         & 41  & 647335685418582083         & 71  & 3793447898643022179662596244366        \\
12  & 194290        & 42  & 1731773886602728051        & 72  & 10072060467737818893614010324770       \\
13  & 534127        & 43  & 4631250509157734047        & 73  & 26738498106822231994902593485746       \\
14  & 1465245       & 44  & 12381460478034483318       & 74  & 70973288195363677225963531535048       \\
15  & 4011126       & 45  & 33092335174560159808       & 75  & 188363010271347363103428460974784      \\
16  & 10961060      & 46  & 88424351052896671941       & 76  & 499851837500292856875731277058977      \\
17  & 29903045      & 47  & 236212572399447537141      & 77  & 1326272602033306568840724593782556     \\
18  & 81429566      & 48  & 630827866930313644489      & 78  & 3518599430142665518024919482166660     \\
19  & 221325445     & 49  & 1684152607151129735036     & 79  & 9333642921927341197807452053383505     \\
20  & 600659520     & 50  & 4494703368297811355435     & 80  & 24755693987767914166837735101399289    \\
21  & 1628709545    & 51  & 11991135688827147388952    & 81  & 65650741449233606049989435056291703    \\
22  & 4414300344    & 52  & 31978416951800296071831    & 82  & 174077611139574752854463708279935997   \\
23  & 11958683448   & 53  & 85250406896754816152086    & 83  & 461512525629540684214148624663888012   \\
24  & 32372736224   & 54  & 227191018857947112334513   & 84  & 1223377600065175892800725928261667064  \\
25  & 87541376014   & 55  & 605282191834901220600054   & 85  & 3242455319972149681281785135048236895  \\
26  & 236440731005  & 56  & 1612185156193460856587117  & 86  & 8592605228187134388298469836076911868  \\
27  & 637862590414  & 57  & 4293176639427000769790008  & 87  & 22767484181294798508811998075481662904 \\
28  & 1719101643609 & 58  & 11430408760122793960003154 &     &                                        \\
29  & 4629525846179 & 59  & 30427812808611490639896278 &     &                                        \\ \hline
\end{tabular}

\bigskip

\caption{Values of $n_{g,2,2}$ for $g \leq 87$.}
\label{tab:ng22}
\end{table}

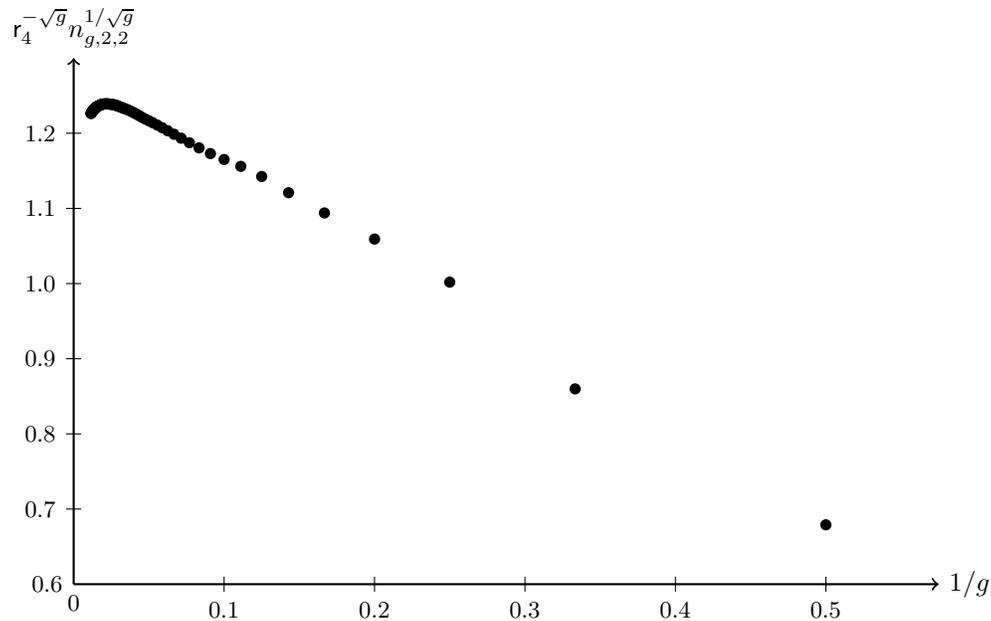
\begin{figure}
  \begin{tikzpicture}[scale=1]
    \foreach \xx/\yy in {0.5/0.679036986368646,0.333333333333333/0.859865301281436,0.25/1.00189460212268,0.2/1.05918597264619,0.166666666666667/1.09392692540852,0.142857142857143/1.12087139436439,0.125/1.14237778747078,0.111111111111111/1.15591866277417,0.1/1.16501247372463,0.0909090909090909/1.17293243248027,0.0833333333333333/1.18051180125345,0.0769230769230769/1.18739262959151,0.0714285714285714/1.19340965399733,0.0666666666666667/1.1986518708796,0.0625/1.20327799595907,0.0588235294117647/1.20736388754395,0.0555555555555556/1.21089757904116,0.0526315789473684/1.21390000915701,0.05/1.21650601519924,0.0476190476190476/1.2189020278834,0.0454545454545455/1.2212067886445,0.0434782608695652/1.22342115746119,0.0416666666666667/1.22547195153905,0.04/1.22728724015724,0.0384615384615385/1.22884124594856,0.037037037037037/1.23015568458936,0.0357142857142857/1.23127894528227,0.0344827586206897/1.23226525080904,0.0333333333333333/1.23316115109611,0.032258064516129/1.23399769289593,0.03125/1.23478715078122,0.0303030303030303/1.23552533031992,0.0294117647058824/1.23619892142693,0.0285714285714286/1.23679431457287,0.0277777777777778/1.23730368631705,0.027027027027027/1.23772658170715,0.0263157894736842/1.23806818756448,0.0256410256410256/1.23833673035961,0.025/1.23854171497365,0.024390243902439/1.23869326976046,0.0238095238095238/1.23880190411737,0.0232558139534884/1.23887792370092,0.0227272727272727/1.23893027111899,0.0222222222222222/1.2389651347124,0.0217391304347826/1.23898493514466,0.0212765957446809/1.23898818361376,0.0208333333333333/1.23897034337429,0.0204081632653061/1.23892543570899,0.02/1.23884789183555,0.0196078431372549/1.23873413259891,0.0192307692307692/1.23858351344552,0.0188679245283019/1.23839849904418,0.0185185185185185/1.2381841361635,0.0181818181818182/1.23794703013373,0.0178571428571429/1.23769409759046,0.0175438596491228/1.2374313788293,0.0172413793103448/1.23716315386224,0.0169491525423729/1.23689152004548,0.0166666666666667/1.2366164686817,0.0163934426229508/1.23633637314798,0.0161290322580645/1.23604871017082,0.0158730158730159/1.23575080695297,0.015625/1.23544044312156,0.0153846153846154/1.23511621539457,0.0151515151515152/1.23477765892974,0.0149253731343284/1.23442518141758,0.0147058823529412/1.23405988973218,0.0144927536231884/1.23368337882914,0.0142857142857143/1.23329752507943,0.0140845070422535/1.23290429943621,0.0138888888888889/1.23250560106196,0.0136986301369863/1.23210311084491,0.0135135135135135/1.23169817070712,0.0133333333333333/1.23129170056252,0.0131578947368421/1.23088416457993,0.012987012987013/1.2304755912924,0.0128205128205128/1.23006564153533,0.0126582278481013/1.22965370922944,0.0125/1.22923903636144,0.0123456790123457/1.22882082585137,0.0121951219512195/1.2283983421331,0.0120481927710843/1.2279709956455,0.0119047619047619/1.22753841105139,0.0117647058823529/1.22710047883588,0.0116279069767442/1.22665738713994,0.0114942528735632/1.226209627779} {
    \node[vtx] at (\xx*20,\yy*10) {};
    }
    \draw[->, thick] (0,6) node[below] {\footnotesize $0$}--(11.5,6) node[right] {\small $1/g$};
    \draw[->, thick] (0,6) node[left] {\footnotesize $0.6$}--(0,13) node[above] {\small $\rr_4^{-\sqrt{g}}n_{g,2,2}^{1/\sqrt{g}}$};
    \draw (-0.1,7) node[left] {\footnotesize$0.7$}--(0.1,7);
    \draw (-0.1,8) node[left] {\footnotesize$0.8$}--(0.1,8);
    \draw (-0.1,9) node[left] {\footnotesize$0.9$} --(0.1,9);
    \draw (-0.1,10) node[left] {\footnotesize$1.0$}--(0.1,10);
    \draw (-0.1,11) node[left] {\footnotesize$1.1$}--(0.1,11);
    \draw (-0.1,12) node[left] {\footnotesize$1.2$}--(0.1,12);
    \foreach \xx in {1,...,5} {
    \draw (2*\xx,6.1)--(2*\xx,5.9) node[below] {\footnotesize 0.\xx};
    }
  \end{tikzpicture}

  \caption{Plot of $\rr_4^{-\sqrt{g}}n_{g,2,2}^{1/\sqrt{g}}$ against $1/g$ for $2 \leq g \leq 87$.}

  \label{fig:plot}
\end{figure}

\subsection{Partition labelings}

In~\Cref{sec:part}, we generalize the notion of Kunz words to partition labelings. In the one-dimensional setting, Kunz words allow us to reinterpret the enumeration of numerical semigroups as a polytopal~\cite{kaplan_2021} and additive-combinatorial~\cite{bacher_2019} problem, which allow us to use tools such as Ehrhart theory~\cite{kaplan_2011} and graph homomorphisms~\cite{li_2022, zhu_2022} to count semigroups.

Thus, it is natural to ask whether these methods can be extended to partition labelings.

\begin{question}
  Can we interpret partition labelings in a polytopal or additive-combinatorial setting? 
\end{question}

\bibliographystyle{general}
\bibliography{general}

\end{document}